\documentclass[11pt]{article}
\usepackage{geometry} 
\usepackage{amsmath,amsthm,amssymb}                
\usepackage{graphicx, color}
\usepackage{amssymb}
\usepackage{epstopdf}
\usepackage{pdfsync}
\usepackage{mathabx}

\newtheorem{definition}{Definition}[section]
\newtheorem{lemma}[definition]{Lemma}
\newtheorem{theorem}[definition]{Theorem}
\newtheorem{proposition}[definition]{Proposition}
\newtheorem{corollary}[definition]{Corollary}
\newtheorem{remark}[definition]{Remark}
\newtheorem{example}[definition]{Example}

\RequirePackage{authblk}

\numberwithin{equation}{section}

\def\e{\varepsilon}

\def\R{\mathbb R}

\def\N{\mathbb N}

\def\AU{\mathcal U}

\def\FF{\mathcal F}

\def\dx{\,dx}
\def\dy{\,dy}
\def\wto{\rightharpoonup}
\def\MM{\mathbb M^{m\times d}}

\newcommand\blfootnote[1]{%
  \begingroup
  \renewcommand\thefootnote{}\footnote{#1}%
  \addtocounter{footnote}{-1}%
  \endgroup
}

\title{A closure theorem for $\Gamma$-convergence and H-convergence\\
with applications to non-periodic homogenization}
\author[1]{Andrea Braides} %\footnote{corresponding author, e-mail: {\tt abraides@sissa.it}} 
\author[1]{Gianni Dal Maso}
\author[2]{Claude Le Bris}
\affil[1]{\small SISSA, via Bonomea 265, Trieste, Italy {\tt \{abraides,dalmaso\}@sissa.it}}
\affil[2]{\small \'Ecole Nationale des Ponts et Chauss\'ees and INRIA, MATHERIALS project-team,
77455~Marne-La-Vall\'ee Cedex 2, France,
{\tt claude.le-bris@enpc.fr}}
\date{}
\begin{document} 

\maketitle

\begin{abstract} In this work we examine the stability of some classes of integrals, in particular with respect to homogenization. The prototypical case is the homogenization of quadratic energies with periodic coefficients perturbed by a term vanishing at infinity, which has been recently examined in the framework of elliptic PDE. We use localization techniques  and higher-integrability Meyers-type results to provide a closure theorem by $\Gamma$-convergence within a large class of integral functionals. From such result we derive stability theorems in homogenization which comprise the case of perturbations with zero average on the whole space. The results are also extended to the stochastic case, and specialized to the $G$-convergence of operators corresponding to quadratic forms. A corresponding analysis is also carried on for non-symmetric operators using the localization properties of $H$-convergence. Finally, we treat the case of perforated domains with Neumann boundary condition, and their stability.



\end{abstract}

\blfootnote{Preprint SISSA  03/2024/MATE}

 \def\OO{\Omega}
 \def\DD{D}

\section{Introduction} In this work we examine the stability of some classes of integrals, and in particular their homogenization. The prototypical case is that of quadratic energies
\begin{equation}\label{intro-1}
\int_\DD \Bigl(a_0\Bigl({x\over\e}\Bigr)+ \widetilde a\Bigl({x\over\e}\Bigr)\Big)|\nabla u|^2\dx,
\end{equation}
with bounded coefficients $a_0,\widetilde a$ such that $a_0$ and $a=a_0+\widetilde a$ are strictly positive, on a bounded subset $\DD$ in $\mathbb R^d$.
Typically $a_0=a_{\rm per}$ is a periodic coefficient, and $\widetilde a$ is considered as a perturbation,
 for which we ask the question, whether we give general conditions on $\widetilde a$ ensuring that such a perturbation is negligible in the problems of homogenization; that is, when $\e$ tends to $0$.
  This issue has been recently widely studied in the framework of the corresponding elliptic operators (see the discussion in Section \ref{PDE}). 
  
 Here, we state the problem as a general stability question for the $\Gamma$-convergence of integral functionals, of which quadratic energies are a particular case. 
In particular, for \eqref{intro-1} we show that a possible condition for $\widetilde a$ to be negligible is that $|\widetilde a|$ has zero mean in $\mathbb R^d$; that is,  
\begin{equation}\label{intro-2}
\langle |\widetilde a|\rangle_{{\mathbb R}^d}:=\lim_{R\to+\infty} {1\over R^d} \int_{B_R}|\widetilde a(y)|\,dy=0. 
\end{equation}
This condition is not necessary (this is evident from the analysis of the case $d=1$, from which we also produce counterexamples in higher dimension), but it is handy in dealing with a number of situations where the periodicity of the functional is perturbed. 

The $\Gamma$-convergence framework gives us the possibility of extending some results in the existing literature for quadratic forms, even though in a weaker form since it is not coupled with a corrector result. These extensions include for example, the treatment of energies for non-linear problems, perforated domains, and random coefficients.

\smallskip
The results we prove are based on the classical localization techniques of $\Gamma$-conver\-gence for integrals 
 of the form 
$$
\int_D f(x,\nabla u)\dx \qquad \hbox{for $u\in W^{1,p}(D;\mathbb R^m)$,} 
$$
with $D$ an open subset of $\mathbb R^d$ and $f=f(x,\xi)\colon D\times \MM\to[0,+\infty)$, where $\MM$ denotes the space of $m\times d$ matrices.
On the one hand, this allows us to obtain compactness and integral-representation results for equicoercive classes of integral functionals of $p$-growth defined on Sobolev spaces (see Theorem \ref{comp}) and on the other hand it allows us to describe the integrands of such functionals in terms of minimum problems on small balls (see Theorems \ref{derit} and \ref{derivation}). These localization techniques, together with Meyers-type higher-integrability estimates, are used to obtain our main result (Theorem \ref{main}): a closure theorem giving a sufficient condition for sequences of functionals to be close in the sense of $\Gamma$-con\-vergence.  In the particular case of two families of integrands $f_k$ and $g_k$  (Corollary \ref{Corgen}), this result asserts that the condition
\begin{equation}\label{intro-001}\nonumber
\lim_{{\rho\to0}}\limsup_{k\to+\infty}{1\over\rho^d}\int_{B_\rho(x)}\sup_{|\xi|\le t} |f_k(y,\xi)-g_k(y,\xi)|\dy=0
\end{equation}
for all $t>0$ and almost all $x\in\DD$ provides a {\em stability property}; namely that, up to subsequences
 \begin{equation}\label{intro-002-bis}\nonumber
 \Gamma\hbox{-}\lim_{k\to +\infty}\int_\DD f_k(x,\nabla u)\dx= 
 \Gamma\hbox{-}\lim_{k\to +\infty}\int_\DD g_k(x,\nabla u)\dx.
 \end{equation}
 In the case of homogenization problems, this condition is further specialized, and reads 
 \begin{equation}\label{intro-002}
 \lim_{R\to+\infty}{1\over R^d}\int_{B_{R}}\sup_{|\xi|\le t} |f(y,\xi)- g(y,\xi)|\dy=0
 \end{equation}
for all $t>0$, in this case ensuring that
 \begin{equation}\label{intro-003}\nonumber
 \Gamma\hbox{-}\lim_{\e\to 0}\int_\DD f\Big(\frac x\e,\nabla u\Big)\dx= 
 \Gamma\hbox{-}\lim_{\e\to 0}\int_\DD g\Big(\frac x\e,\nabla u\Big)\dx
 \end{equation}
 if either of the two limits exists. 
 
 Those stability results extend to the stochastic case, for which condition \eqref{intro-002} is now expressed as
 a hypothesis involving expectations, of the form
 \begin{equation}\label{intro-004}\nonumber
\lim_{R\to+\infty}{1\over R^d}\int_{\OO}
 \bigg(\int_{B_{R}}\sup_{|\xi|\le t} |f(\omega,y,\xi)- g(\omega, y,\xi)|\dy\bigg)\,d P(\omega)
=0
\end{equation}
 for all $t>0$, where $({\OO},{\mathcal T},P)$ denotes the probability space in which the stochastic problem is formulated. In this case the almost sure homogenizability of $f$, which is usually obtained by statistical invariance and ergodic assumptions, implies the  homogenizability of $g$ (see Section \ref{stock}).
 
When we restrict to quadratic functionals of the form
\[
\int_\DD \langle A(x)\nabla u,\nabla u\rangle\dx;
\]
that is, $f(x,\xi)=\langle A(x)\xi,\xi\rangle$, $\Gamma$-convergence results can be translated into the corresponding $G$-convergence results regarding the behaviour of elliptic problems of the form
  \begin{equation}\label{intro-006}
\begin{cases} -{\rm div}(A_k\nabla u_k)=\phi \\
u_k\in H^1_0(\DD).
\end{cases}
\end{equation}
 In this case, the condition
 \begin{equation}\label{intro-007}\nonumber
\lim_{{\rho\to0}}\limsup_{k\to+\infty}{1\over\rho^d}\int_{B_\rho(x)} |A_k(y)-B_k(y)|\dy=0
\end{equation}
for almost every $x$
ensures that the $G$-limits of the operators corresponding to $A_k$ and $B_k$ are the same. In the case of homogenization, where $A_k(x)=A(x/\e_k)$
and $B_k(x)=B(x/\e_k)$, a condition providing stability is
\begin{equation}\label{intro-008} \nonumber
\limsup_{R\to+\infty}{1\over R^d}\int_{B_{R}} |A(x)- B(x)|\dx
=0,
 \end{equation}
which reduces to \eqref{intro-2} in the case of isotropic matrices as in \eqref{intro-1}. 
Similar conditions and results can be stated also in the stochastic case. 

We can also address the case of $H$-convergence; that is, convergence of solutions to problems of the form \eqref{intro-006} for non-symmetric $A_k$. To this end, we adapt the localization techniques introduced for minimizers (Section \ref{closure-sec}) to the case of solutions of elliptic equations.

The final part of the article is dedicated to problems in perforated domains with Neumann boundary conditions, a question related to the behaviour as $\e\to0$ of functionals of the form
\begin{equation}\label{intro-010}\nonumber
F^E_\e(u)= \int_{\mathbb R^d\setminus \e E} |\nabla u|^2\dx.
\end{equation}
The analysis there is focussed on the perturbations of the perforation set $E$ that ensure stability of the corresponding $\Gamma$-limits.

\section{Comparison with the PDE approach}\label{PDE}
We comment here upon the relation of the present work with some existing results in the literature of the theory of partial differential equations. In the context of elliptic equations and systems (first in divergence form and next for some other cases), homogenization theory in the presence of perturbations in an otherwise periodic structure was developed in the work~\cite{MR3000490} by X.~Blanc, P.-L.~Lions and the third author, and  in the subsequent works~\cite{MR3421758,MR3909031,MR4060222,MR3926042}. In this series of works, which all follow a PDE perspective, some necessary assumptions are limiting the applicability of the results. These assumptions essentially are

\smallskip
(1) only  \emph{localized}  perturbations are considered; that is, perturbation coefficients~$\tilde a$ in \eqref{intro-1} that in some sense vanish at infinity (this will be made precise below);

\smallskip
(2) the unperturbed coefficient $a_0$  in \eqref{intro-1} is assumed periodic (this is related to the technique of proof, which actually relies on earlier results by M.~Avellaneda and F.-H.~Lin in the seminal work~\cite{MR0910954}, complemented by~\cite{MR0978702,MR1010728,MR1127038}, all precisely on the periodic case);

\smallskip (3) for reasons again related to the very technique of proof the (uniform in space) H\"older continuity of both the unperturbed coefficient and the perturbation coefficient (thus of their sum) needs to be assumed. 

The two exceptions to (3) are, on the one hand, the Hilbertian setting, when it is assumed that~$\tilde a\in L^2({\mathbb R}^d)$ and for which the H\"older continuity of the coefficients is unnecessary because everything is taken care of by ellipticity, and, on the other hand, the case of \emph{equations} instead of \emph{systems}, for which the Nash-Moser regularity theory allows to establish some results again without assuming the H\"older continuity.

\smallskip

In order to formalize (1); that is, the locality of the perturbation, the natural assumption on the  perturbation coefficients $\tilde a$ would be, besides the above assumptions (2)-(3), that $\tilde a$ only converges to zero at infinity; that is, $\displaystyle \tilde a(x)\to 0$ as $|x|\to+\infty$. This ``natural" setting is however not tractable mathematically.
\textit{A fortiori},  the (weaker) assumption~\eqref{intro-2}, although very natural, in particular because it mimics what happen for averages of functions, cannot be handled as such.
Hence,   the ``locality" was encoded  in imposing the stronger assumption
$$\tilde a\in L^p({\mathbb R}^d), \quad \textrm{for some}\quad 1\leq p< +\infty.$$
Using H\"older's inequality, this assumption clearly implies condition~\eqref{intro-2} on the vanishing average.

Given these restricting assumptions, the works~\cite{MR3000490,MR3421758,MR3909031,MR4060222} 

(i) establish that the homogenized coefficient is left unchanged by the perturbation,

(ii) prove the existence of a suitable corrector function, adapted to the perturbation, 

(iii) identify a specific rate of convergence for the residual (that is, $u^\varepsilon$ minus the first two terms of the two-scale expansion), which  may be made precise depending upon the functional space  (Sobolev spaces and H\"older spaces) considered for measuring the residual, the ambient dimension $d$,  and the exponent $p$ of the Lebesgue space in the condition $\tilde a\in L^p({\mathbb R}^d)$ on the perturbation coefficient.

\smallskip

So, as compared to the results using the $\Gamma$-limit approach and developed in the present article, the works~\cite{MR3000490,MR3421758,MR3909031,MR4060222}  use much stronger assumptions, but prove more on the homogenization limit. They also extend to some equations and systems that are not in divergence form and/or are not self-adjoint, as the work~\cite{MR3926042} shows. That said, as opposed to the approach developed in the present work, the techniques of the works~\cite{MR3000490,MR3421758,MR3909031,MR4060222} do not seem to allow one to establish (i) under the only assumption~\eqref{intro-2}.
\smallskip

In a related work~\cite{MR4459619} by R.~Goudey,  a variant of the above ``local" regime was explored. 
The idea therein is easily illustrated by the particular case when the perturbation~$\tilde a$ is in fact supported in the vicinity of the set of points in~${\mathbb R}^d$ that have coordinates along the canonical axes that all are integer powers of, say, 2. Intuitively, such points become more and more rare at infinity but never fully disappear. 
Such a perturbation ``rare at infinity'' also satisfies condition \eqref{intro-2} and so, formally, the setting of~\cite{MR4459619} is a particular instance of that considered in the present work.

\smallskip
A motivation for the works~\cite{MR3000490,MR3421758,MR3909031,MR4060222}  arises from a preceding work~\cite{MR1974463} on nonperiodic geometries. 
The underlying formalism for the theory was then originally introduced  in a slightly different (but intrinsically related) context, that of  \textit{thermodynamic limit problems}. A suitable set of points~$\displaystyle \{X_k\}_{{\mathbb Z}^d}$, distributed over the ambient space~${\mathbb R}^d$, is first considered. The points are not necessarily arranged in a periodic array, but are sufficiently well organized geometrically. 
Prototypical functions are next constructed using translations along this set of points; that is, functions of the form
$ \sum_{k\in{\mathbb  Z}^d} \psi (x-X_k)$, 
 for~$\psi\in C^\infty_0({\mathbb R}^d)$. If some adequate geometric conditions 
are imposed on the locations of the points as well as on the two-point, three-point etc, correlations, then it is possible to construct some \emph{algebras}~$\mathcal A$ of functions that have interesting averaging properties and may be used for thermodynamic limit problems. 
 In a subsequent work, the construction was adapted to the context of homogenization theory, see Assumptions (H1)-(H2)-(H3$'$) in~\cite[Section 5.3]{MR2334772}.
Omitting some technicalities, the question of developing a homogenization theory for  equations  with coefficients~$a$ in such an algebra~$\mathcal A$  then reduces to establishing the existence of a solution~$w_p$ to  the corrector equation in this algebra. 
It is again interesting to note that the first geometric condition imposed in~\cite{MR1974463,MR2334772} ensures that the number of points within an arbitrary  ball grows at most linearly with respect to the volume of the ball. Put differently, this is a condition on the \textit{average} of the coefficient and is therefore close in spirit to condition~\eqref{intro-2} used in the present work. For a recent example of a work in this line of thought, we refer to~\cite{siam-goudey}.
\smallskip
 
 Finally, we mention that a \emph{quasilinear} setting was explored in the work~\cite{MR4523487} by S.~Wolf using the techniques introduced in the series of works mentioned above. The equation considered in~\cite{MR4523487} involves the $p$-Laplacian. It reads
 $$-\,\textrm{div}\left(a\Big({x\over\varepsilon}\Big)\,|\nabla u_\varepsilon|^{p-2}\,\nabla u_\varepsilon\right)\,=\,f(x),$$
where the coefficient~$a$ is again as above. The present article will show in Section \ref{closure-sec} that this case can be treated by $\Gamma$-convergence if \eqref{intro-2} holds.

\section{Notation and preliminaries}\label{pre-sec}
The fixed dimensions of a reference and target Euclidean space, respectively, will be denoted by  $d\ge1$ and $m\ge 1$. The space of $m\times d$ matrices is denoted by $\MM$. If $\xi\in \MM$ then $\ell_\xi$ denotes the linear function $x\mapsto \xi x$.
If $\DD$ is an open subset of $\mathbb R^d$ the family of all bounded open subsets of $\DD$ will be denoted by $\AU=\AU(\DD)$. The open ball of centre $x$ and radius $r$ is denoted by $B_r(x)$. We omit $x$ when $x=0$. The notation $1_U$ is used for the characteristic function of the set $U$.

\smallskip
In the following $p>1$ will always be a fixed exponent, and $\alpha,\beta>0$ with $\alpha\le \beta$ two fixed constants. With fixed  $\DD$ an open subset of $\mathbb R^d$ we will consider the class $\mathcal F=\mathcal F_{\alpha,\beta}(\DD)$ of all Borel functions $f\colon \DD\times\MM\to [0,+\infty)$ satisfying
\begin{equation}\label{(16.1)}
\alpha|\xi|^p\le f(x,\xi)\le \beta(1+|\xi|^p) 
\end{equation}
for almost all $x\in\DD$ and $\xi\in\MM${, with $f(x,\cdot)$ quasiconvex for almost all $x$. We recall that quasiconvexity implies the local Lipschitz condition 
\begin{equation}\label{lipcond}
|f(x,\xi)- f(x,\xi')|\le C(1+|\xi|^{p-1}+|\xi'|^{p-1})|\xi-\xi'|
\end{equation}
for all $\xi,\xi'\in\MM$, with $C$ depending only on $\alpha,\beta,d$  and $m$.
}

If $f\in\mathcal F$, given $u\in W^{1,p}(\DD;\mathbb R^m)$ and  $U\in\AU$ we will set 
\begin{equation}\label{Ff}
F(u,U)=\int_U f(x,\nabla u)\dx.
\end{equation}
Since we will consider $\Gamma$-limits of functionals of the type $F$, the assumption that $f$ is quasiconvex is not restrictive, up to a relaxation argument (see \cite{DM,GCB}), since we may assume that $F(\cdot,{U})$ is weakly lower semicontinuous in $W^{1,p}(U;\mathbb R^m)$ for $U\in\mathcal U$, which is equivalent to $f$ being quasiconvex (in the second variable) (see \cite{BDF,MR2361288}). This will be done throughout the article. Note that \eqref{(16.1)} and  \eqref{lipcond} imply in particular the continuity of $F$ in the strong topology of $W^{1,p}(U;\mathbb R^m)$.

If $f\in\mathcal F$ and $F$ is the related functional, then the values of $f$ can be recovered from minimum problems regarding $F$.
Indeed, the following theorem holds  (see \cite{DMM}).

\begin{theorem}\label{derit} Let $f\in\mathcal F$ and let $F$ be the related functional; then for almost all $x\in\DD$ and for all $\xi\in\MM$
the limit in \eqref{rap-f} exists and we have
\begin{equation}\label{rap-f}
f(x,\xi)=\lim_{\rho\to 0} {1\over |B_\rho|}\min\big\{F(v,B_\rho(x)): v\in\ell_\xi +W^{1,p}_0(B_\rho(x);\mathbb R^m)\big\}.
\end{equation}
\end{theorem}
The well-known compactness properties of functionals with integrands $f\in\mathcal F$ is stated in the following theorem \cite{DM,BDF}.

\begin{theorem}[Compactness properties of the class $\mathcal F$]\label{comp}
Given $\DD$ open subset of $\mathbb R^d$ and $f_k\in\mathcal F$ there exists a subsequence (still denoted by $f_k$) such that for all $U\in\AU$ the functionals $F_k(\cdot, U)$ given by
\begin{equation}\label{Ff_k}
F_k(u,U)=\int_U f_k(x,\nabla u)\dx
\end{equation} 
$\Gamma$-converge  on $W^{1,p}(\DD;\mathbb R^m)$ with respect to the weak convergence in $W^{1,p}(\DD;\mathbb R^m)$  to $F_\infty(\cdot, U)$, with $F_\infty$ given by \eqref{Ff} for some $f_\infty\in \mathcal F$.
Furthermore, this class enjoys the Urysohn property; that is, $F_k(\cdot, U)$ $\Gamma$-converges to $F_\infty(\cdot, U)$ for all $U\in\AU$  if and only if for all subsequence $F_{k_j}(\cdot, U)$ there exists a further subsequence which $\Gamma$-converges to $F_\infty(\cdot, U)$  for all $U\in\AU$. 
\end{theorem}

\begin{remark}[Extension to unbounded domains]\label{unb-comp}\rm 
We can extend Theorem \ref{comp} to  $U=D$, even if $D$ is not bounded,  upon assuming that
$f_k(x,\xi)\le \beta (a(x)+|\xi|^p)$ for some $a\in L^1(D)$ and for all $k$.
 More precisely, we can show that if $f_\infty$ is the integrand given by Theorem \ref{comp}, then the functional
$$
F_\infty(u)=\int_D f_\infty(x,\nabla u)\dx
$$
is the $\Gamma$-limit of the sequence 
$$
F_k(u)=\int_D f_k(x,\nabla u)\dx
$$
in the weak topology of $W^{1,p}(D;\mathbb R^m)$. 

The $\Gamma$-liminf inequality for $D$ follows from the corresponding inequality for $U\in\AU$, noting that  $F_k(u)\ge F_k(u,U)$,  which is a consequence of the positiveness of $f_k$.
As for the $\Gamma$-limsup inequality, we recall that by the continuity of $F_\infty(\cdot)$ in $W^{1,p}(D;\mathbb R^m)$ it is sufficient to construct a recovery sequence when $ u\in W^{1,p}(D;\mathbb R^m)$ has compact support in $\mathbb R^d$, because of the strong density of these functions in $W^{1,p}(D;\mathbb R^m)$. By taking $U=B_{R}\cap D$, where $B_R$ is a ball containing the support of $u$, Theorem \ref{comp} and a cut-off argument used to match the boundary conditions \cite[Chapter 21]{DM} we can find a recovery sequence $u_k$ for $F_\infty(u,U)$ in $W^{1,p}(U;\mathbb R^m)$ with $u_k=0$ on $D\cap \partial B_R$. Recalling that $f_k(x,0)=0$ for all $k$ and almost every $x$, the extension to $0$ of $u_k$ on $D\setminus B_R$ provides a recovery sequence for $F_\infty(u)$.
\end{remark}

The following derivation theorem allows us to characterize $\Gamma$-convergence through the convergence of minimum problems \cite{DMM}.

\begin{theorem}[A characterization of integrands of $\Gamma$-limits]\label{derivation}
 If for almost all $x\in\DD$ and for all $\xi\in\MM$ we have
\begin{eqnarray}
f_\infty(x,\xi)=\limsup_{\rho\to 0} \limsup_{k\to+\infty}{1\over |B_\rho|}\min\big\{F_k(v,B_\rho(x)): v\in\ell_\xi +W^{1,p}_0(B_\rho(x);\mathbb R^m)\big\}
\\
=\limsup_{\rho\to 0} \liminf_{k\to+\infty}{1\over |B_\rho|}\min\big\{F_k(v,B_\rho(x)): v\in\ell_\xi +W^{1,p}_0(B_\rho(x);\mathbb R^m)\big\},
\end{eqnarray}
then for every $U\in\AU$  the sequence $F_k(\cdot, U)\hbox{ $\Gamma$-converges to }F_\infty(\cdot, U)$ given by \eqref{Ff} with $f=f_\infty$. Conversely, if $F_k(\cdot, U)$ $\Gamma$-con\-verges to $F_\infty(\cdot, U)$ with integrand $f_\infty\in\FF$ for all $U\in\AU$ then for every such $U$ there exists the limit  
$$
\lim_{k\to+\infty}\min\big\{F_k(v,U): v\in\ell_\xi +W^{1,p}_0(U;\mathbb R^m)\big\},
$$
and we have 
\begin{eqnarray}\label{derivfk}
f_\infty(x,\xi)=\lim_{\rho\to 0} \lim_{k\to+\infty}{1\over |B_\rho|}\min\big\{F_k(v,B_\rho(x)): v\in\ell_\xi +W^{1,p}_0(B_\rho(x);\mathbb R^m)\big\}
\end{eqnarray}
for almost all $x\in\DD$ and for all $\xi\in\MM$.
\end{theorem}

Finally, we recall a particular case of a higher-integrability theorem by Meyers and Elcrat \cite{M-E} (see also \cite[Theorem 1]{gallouet} for the result in the case of an equation instead of a system).

\begin{theorem}\label{16.1}
{\rm (Meyers Regularity Theorem)} 
   \index{Meyers Regularity Theorem}
There exists $\eta>0$ depending only on $\alpha,\beta$ such that for all $f\in\FF$, $B$ an open ball contained in $D$, and $\xi\in\MM$, the minimum points $u$ of
\begin{equation}\label{(16.2)}
F(u,B)=\int_B f(x,\nabla u)\dx
\end{equation}
on $\ell_\xi+W^{1,p}_0(B;\mathbb R^m)$ belong to $W^{1,p+\eta}(B;\mathbb R^m)$, and
\begin{equation}\label{(16.3)}
{1\over |B|}\int_{B}|\nabla u|^{p+\eta}\dx\le C({  1+ }|\xi|^{p+\eta}),
\end{equation}
with the constant $C$ depending only on $\alpha,\beta$, but not on $\xi$ and $B$.
\end{theorem}

The proof follows from \cite{M-E}, upon rewriting minimum problems on a common ball $B_1$ and with 
boundary data given by $\ell_{\xi/|\xi|}$ (unless $\xi=0$, in which case the boundary datum is $0$) by a change of variables and noting that inequalities \eqref{(16.1)} are maintained after change of variables.

\section{A closure theorem for $\Gamma$-convergence}\label{closure-sec}
The following result provides a generalization and localization of the closure theorem in \cite{B1} (see also \cite[Section 17]{BDF}), with an independent proof obtained by using the derivation formula in Theorem \ref{derivation}.

  \begin{theorem}\label{main}
Let $f_k, g^j_k$ be functions in $\FF$, and let 
$F_k,G^j_k$ be the corresponding integral functionals.  
Assume that  for almost all $x\in\DD$ we have
\begin{equation}\label{hypo-z}
\lim_{j\to+\infty}\limsup_{  \rho\to0}\ 
\limsup_{k\to+\infty}{1\over\rho^d}\int_{B_\rho(x)}\sup_{|\xi|\le t} |f_k(y,\xi)-g^j_k(y,\xi)|\dy=0
\end{equation}
for all $t>0$. 
If for every $j$ the sequence  $G^j_k(\cdot, U)$ $\Gamma$-converges to $G^j_\infty(\cdot, U)$ as $k\to+\infty$ for every $U\in\AU$ 
and $g^j_\infty\in \FF$ is the corresponding integrand, then there exists the pointwise limit 
$g_\infty(x,\xi)$ of $g^j_\infty(x,\xi)$ for almost every $x\in\DD$ and $\xi\in\MM$, and
$F_k(\cdot, U)$ $\Gamma$-converges to $G_\infty(\cdot, U)$
for all such $U$, where $G_\infty$ is the integral functional corresponding to $g_\infty$.
 \end{theorem}

\begin{proof} For every $j$, we assume that $G^j_k(\cdot, U)$ $\Gamma$-converges to $G^j_\infty(\cdot, U)$ as $k\to+\infty$, and, by the compactness Theorem \ref{comp} we may also assume that $F_k(\cdot, U)$ $\Gamma$-converges to $F_\infty(\cdot, U)$ for every $U\in\AU$. Let $g^j_\infty,f_\infty$ be the integrands of $G^j_\infty$ and $F_\infty$, respectively. By the Urysohn property of $\Gamma$-convergence \cite[Proposition 8.3]{DM} to conclude the proof it suffices to show that 
there exists the limit of $g^j_\infty(x,\xi)$ as $j\to+\infty$ and 
\begin{equation}\label{claim1}
\hbox{$\lim\limits_{j\to+\infty}g^j_\infty(x,\xi)=f_\infty(x,\xi)$ for almost all $x\in\DD$ and for all $\xi\in\MM$.}
\end{equation}

For every $\Phi\in\mathcal F$, $\rho>0$, $\xi\in\MM$, and $x\in\DD$, we set
$$
m^\Phi(x,\rho,\xi)=\min\big\{\Phi(v,B_\rho(x)): v\in\ell_\xi +W^{1,p}_0(B_\rho(x))\big\}.
$$
We fix $\xi_0\in\mathbb R$ and  $x_0\in\DD$ such that \eqref{derivfk} holds for $G^j_k$ and $g^j_\infty$ for all $j$, and for $F_k$ and $f_\infty$. We wish to estimate $m^{F_k}(x_0,\rho,\xi_0)- m^{G^j_k}(x_0,\rho,\xi_0)$.

With fixed $j$ and $k$, let $u$ be a minimizer of the problem defining $m^{G^j_k}(x_0,\rho,\xi_0)$. We set 
\begin{equation}\label{defpsijk}
\psi^j_k(x,\xi)=|g^j_k(x,\xi)-f_k(x,\xi)|,
\end{equation} so that we have the estimate
\begin{equation}\label{2.10}
m^{F_k}(x_0,\rho,\xi_0)\le m^{G^j_k}(x_0,\rho,\xi_0)+\int_{B_\rho(x_0)} \psi^j_k(y,\nabla u)\dy.
\end{equation}
For all $t>0$ let $E_t=\{y\in B_\rho(x_0): |\nabla u(y)|>t\}$. By Chebyshev's inequality we have
\begin{equation}
|E_t|\le{1\over t^p}\int_{B_\rho(x_0)}|\nabla u|^p\dy\le{1\over \alpha t^p}G^j_k(u,B_\rho(x_0))\le {\beta\over \alpha t^p}
|B_\rho|(1+|\xi_0|^p).
\end{equation}
 By Theorem \ref{16.1} there exist  $\eta>0$ and $C>0$, depending only on $\alpha$ and $\beta$, such that \eqref{(16.3)} holds, so that, by H\"older's inequality
\begin{equation}
\int_{E_t}|\nabla u|^p\dy\le |E_t|^{\eta\over p+\eta}\Big(\int_{B_\rho(x_0)}|\nabla u|^{p+\eta}\dy\Big)^{p\over p+\eta}\le C_1|B_\rho| t^{-{\eta p\over p+\eta}}
\end{equation}
and
\begin{eqnarray*}
\int_{B_\rho(x_0)} \psi^j_k(y,\nabla u)\dy&\le& \int_{E_t} \psi^j_k(y,\nabla u)\dy + \int_{B_\rho(x_0)} \sup_{|\xi|\le t}\psi^j_k(y,\xi)\dy\\
&\le &  \int_{E_t} 2\beta(1+|\nabla u|^p)\dy+ \int_{B_\rho(x_0)} \sup_{|\xi|\le t}\psi^j_k(y,\xi)\dy\\
&\le &  C_2|B_\rho|\Big({1\over t^p}+ {1\over t^{{\eta p\over p+\eta}}}\Big)+ \int_{B_\rho(x_0)} \sup_{|\xi|\le t}\psi^j_k(y,\xi)\dy\
\end{eqnarray*}
for some $C_1$ and $C_2$ depending on $\alpha,\beta$ and $\xi_0$.
Letting first $k\to+\infty$ we get
\begin{eqnarray}\nonumber
\limsup_{k\to+\infty}{1\over |B_\rho|}\int_{B_\rho(x_0)} \psi^j_k(y,\nabla u)\dy\le C_2\Big({1\over t^p}+ {1\over t^{{\eta p\over p+\eta}}}\Big)\ \hskip3.2cm\ &&\\
+ \limsup_{k\to+\infty}{1\over |B_\rho|}\int_{B_\rho(x_0)} \sup_{|\xi|\le t}\psi^j_k(y,\xi)\dy.&&
\end{eqnarray}
Using \eqref{hypo-z} we obtain
\begin{eqnarray}\nonumber
\limsup_{j\to+\infty}\limsup_{\substack {\rho\to0\\ \rho\in E}}\limsup_{k\to+\infty}{1\over |B_\rho|}\int_{B_\rho(x_0)} \psi^j_k(y,\nabla u)\dy\le C_2\Big({1\over t^p}+ {1\over t^{{\eta p\over p+\eta}}}\Big).
\end{eqnarray}
so that, by \eqref{2.10} and \eqref{derivfk} for $f_\infty$ and $g^j_\infty$, we have
\begin{eqnarray*}
\lim_{\rho\to0}\lim_{k\to+\infty}{1\over|B_\rho|}m^{F_k}(x_0,\rho,\xi_0)
\le \liminf_{j\to+\infty} \lim_{\rho\to0}\lim_{k\to+\infty}{1\over|B_\rho|}m^{G^j_k}(x_0,\rho,\xi_0)+ C_2\Big({1\over t^p}+ {1\over t^{{\eta p\over p+\eta}}}\Big).
\end{eqnarray*}
Hence, after letting $t$ tend to $+\infty$, by \eqref{derivfk} (and the analog with $g_\infty$ and $F_k$) we have
\begin{equation}\label{ginf-sup}
f_\infty(x_0,\xi_0)\le \liminf_{j\to+\infty}g^j_\infty(x_0,\xi_0).
\end{equation}

To prove the opposite inequality we can proceed symmetrically, taking now $u$ a minimizer of the problem defining $m^{F_k}(x_0,\rho,\xi_0)$, which gives
\begin{equation}\label{2.10-s}
m^{G^j_k}(x_0,\rho,\xi_0)\le m^{F_k}(x_0,\rho,\xi_0)+\int_{B_\rho(x_0)} \psi^j_k(y,\nabla u)\dy
\end{equation} 
in the place of \eqref{2.10}. The same arguments then lead to 
\begin{equation*}
 \limsup_{j\to+\infty} \lim_{\rho\to0}\lim_{k\to+\infty}{1\over|B_\rho|}m^{G^j_k}(x_0,\rho,\xi_0)\le\lim_{\rho\to0}\lim_{k\to+\infty}{1\over|B_\rho|}m^{F_k}(x_0,\rho,\xi_0)+ C_2\Big({1\over t^p}+ {1\over t^{{\eta p\over p+\eta}}}\Big),
\end{equation*}
and eventually to
\begin{equation}\label{ginf-inf}\limsup_{j\to+\infty}g^j_\infty(x_0,\xi_0)\le f_\infty(x_0,\xi_0) .
\end{equation}
Finally, \eqref{ginf-sup} and \eqref{ginf-inf} prove claim \eqref{claim1}.
\end{proof}

As a particular case of Theorem \ref{main}, we have the following corollary, which provides a sufficient condition for two sequences to have the same $\Gamma$-limit.

\begin{corollary}\label{Corgen}
Let $f_k, g_k$ be two sequences of functions in $\FF$, and let 
$F_k,G_k$ be the corresponding integral functionals.  
Assume that $F_k(\cdot, U)$ $\Gamma$-converges to $F_\infty(\cdot, U)$ for every $U\in\AU$.
Assume also that   for almost all $x\in\DD$ we have
\begin{equation}\label{hypo}
\lim_{ \rho\to0}\limsup_{k\to+\infty}{1\over\rho^d}\int_{B_\rho(x)}\sup_{|\xi|\le t} |f_k(y,\xi)-g_k(y,\xi)|\dy=0
\end{equation}
for all $t>0$. 
Then
$G_k(\cdot, U)$ $\Gamma$-converges to $F_\infty(\cdot, U)$
for all $U\in\AU$.
 \end{corollary}
 
We give a result of independent interest that is a consequence of Theorem \ref{main}.
It provides a new proof of a known property of $\Gamma$-convergence of sequences 
of functionals with pointwise converging integrands. 
 
 \begin{proposition} Let $f_k$ be a sequence in $\FF$ and $f_\infty\in\FF$, and let $F_k$ and $F_\infty$ be the corresponding functionals. Assume that we have
 \begin{equation}
 \lim_{k\to+\infty} f_k(x,\xi)=f_\infty(x,\xi)
 \hbox{ for almost every $x\in\DD$ and $\xi\in\MM$.}
 \end{equation}
 Then $F_k(\cdot, U)$ $\Gamma$-converges to $F_\infty(\cdot, U)$ for every $U\in\AU$ .
 \end{proposition}
 
 \begin{proof} For every $j$ we define $h_j$ as the convex envelope of  
 $$ 
 \begin{cases} 0 & \hbox{ if $|\xi|\le j$}\cr
 \beta(1+|\xi|^p) &\hbox{ if $|\xi|> j$.}
 \end{cases}
 $$
 Note that  there exists a constant $R_j>j$ such that $h_j(\xi)=\beta(1+|\xi|^p)$ if $|\xi|\ge R_j$.
 We define $g^j_k\in\FF$ as
 $
 g^j_k(x,\xi)= \max\{f_k(x,\xi), h_j(\xi)\}$,
 which converges pointwise to $g^j_\infty(x,\xi):=\max\{f_\infty(x,\xi), h_j(\xi)\}$ as $k\to+\infty$. Note that $g^j_\infty$ converge pointwise to $f_\infty$; hence, the corresponding integrals satisfy
 \begin{equation}\label{pointw}
 \lim_{j\to+\infty} G^j_\infty(u,U)= F_\infty(u,U)
 \end{equation} 
 for every $u\in W^{1,p}(\DD;\mathbb R^m)$ and $U\in\AU$.
 
 By the equi-Lipschitz condition \eqref{lipcond} satisfied by functions in $\FF$  we have
 \begin{equation}
 g^j_k(x,\cdot) \to g^j_\infty(x,\cdot) \hbox{ uniformly in }\overline {B_{R_j}}
 \end{equation}
 for almast all $x\in D$. 
Let $\eta^j_k$ be the non-negative functions defined by
 $$
 \eta^j_k(z)=\sup_{|\xi\le R_j}|g^j_k(x,\xi)-g^j_\infty(x,\xi)|.
 $$
Such functions satisfy $\eta^j_k(x)\le 2\beta(1+|R_j|^p)$, converge to $0$ for almost every $x$ and we have 
  \begin{equation}\label{gjxi}
 g^j_k(x,\xi) \ge g^j_\infty(x,\xi)-\eta^j_k(x)\hbox{ for almost every }x\in\DD \hbox{ and } |\xi|\le R_j.
 \end{equation}
 Since $g^j_k(x,\xi)=\beta(1+|\xi|^p)=g^j_\infty(x,\xi)$ for $|\xi|>R_j$ we conclude that \eqref{gjxi} holds for all $\xi$. Hence the corresponding integrals satisfy
  \begin{equation}\label{gjxi-f}
G^j_k(u,U) \ge G^j_\infty(u,U)-\int_U\eta^j_k\dx
 \end{equation} 
 for every $u\in W^{1,p}(\DD;\mathbb R^m)$ and $U\in\AU$.

Since $G^j_\infty(\cdot,U)$ is lower semicontinuous, this implies that $\Gamma$-$\liminf_{k\to+\infty} G^j_k(\cdot,U)\ge G^j_\infty(\cdot,U)$. Since the pointwise limit of $G^j_k(\cdot,U)$ is $G^j_\infty(\cdot,U)$ as $k\to+\infty$, we have the $\Gamma$-limsup inequality, and we obtain the $\Gamma$-convergence of $G^j_k(\cdot,U)$ to $G^j_\infty(\cdot,U)$.
 
 Trivially, $|f_k(x,\xi)-g^j_k(x,\xi)|=0$ if $|\xi|\le j$, so that \eqref{hypo-z} holds. By Theorem \ref{main} and \eqref{pointw} we can conclude that $F_k$ $\Gamma$-converge to $F_\infty$. \end{proof}

 \section{Homogenization problems}\label{hom-sec}
 In this section we consider homogenization problems, when the integrands $f_k$ can be obtained by scaling a single integrand defined on the whole space.

 \subsection{Deterministic homogenization}
 We suppose that
there exists $f\in\FF_{\alpha,\beta}(\mathbb R^d)$ and $\e_k$, a positive sequence tending to $0$, such that
 \begin{equation}
 f_k(x,\xi)= f\Big({x\over\e_k},\xi\Big).
 \end{equation}
In this case $f_k\in\FF$ and the functionals $F_k$ take the form
  \begin{equation}\label{Fekk}
F_k(u,U)=\int_U f\Big({x\over\e_k},\nabla u\Big)\dx.
 \end{equation}
We are interested in the case when $F_k(\cdot,U)$ has a $\Gamma$-limit $F_\infty(\cdot,U)$ for all $U\in\AU$ and the corresponding integrand $f_\infty$ does not depend on $x$  nor on the sequence $\e_k$. We then say that $f$ is {\em homogenizable}, and $F_\infty$ and $f_\infty$ are denoted by $F_{\rm hom}$ and $f_{\rm hom}$, respectively.

Theorem \ref{derivation} provides a characterization of homogenizability thanks to \eqref{derivfk}.
Indeed, by a change of variables $f$ is homogenizable if and only if for almost all $x\in\mathbb R^d$ the limit
\begin{eqnarray}\label{derivfk-hom}
\lim_{\rho\to 0} \lim_{k\to+\infty}{1\over |B_{\frac{\rho}{\e_k}}|}\min\bigg\{\int_{B_{\frac{\rho}{\e_k}}(\frac{x}{\e_k})} f (y,\nabla v )\dy
: v\in\ell_\xi +W^{1,p}_0(B_{\frac{\rho}{\e_k}}(\tfrac{x}{\e_k});\mathbb R^m)\bigg\}
\end{eqnarray}
exists for all $\xi\in\MM$ and is independent of $x$  and of the sequence $\e_k$. This holds in the particular case when $f$ is periodic in the first variable \cite{DM,BDF}.

We now consider a sequence of functions $g^j\in\FF_{\alpha,\beta}(\mathbb R^d)$ and the corresponding  \begin{equation}
 g^j_k(x,\xi)= g^j\Big({x\over\e_k},\xi\Big).
 \end{equation}
The related functions $\psi^j_k$ have the form $\psi^j_k(x,\xi)=\psi^j({x\over \e_k},\xi)$,
where 
\begin{equation}\label{spi=per}
\psi^j(y,\xi)= |f(y,\xi)- g^j(y,\xi)|,
 \end{equation}
 and by a change of variables hypothesis \eqref{hypo} becomes
   \begin{eqnarray}\label{hypo-4}
   \lim_{j\to+\infty}\limsup_{\rho\to0}\limsup_{k\to+\infty}{\e_k^d\over\rho^d}\int_{B_{\rho/\e_k}(x_0/\e_k)}\sup_{|\xi|\le t} \psi^j(y,\xi)\dy=0
\end{eqnarray}
for almost all $x_0$. This hypothesis is satisfied if
\begin{equation}\label{hypo-3} 
\lim_{j\to+\infty}\limsup_{R\to+\infty}{1\over R^d}\int_{B_{R}}\sup_{|\xi|\le t} \psi^j(y,\xi)\dy
=0
 \end{equation}
 for all $t>0$. This observation proves the following result, which can also be found in \cite{B1}.
 
  \begin{theorem}[Homogenizability by approximation]\label{th-hom-j} Let $f\in\FF_{\alpha,\beta}(\mathbb R^d)$, and let $g^j$ be a sequence in $\FF_{\alpha,\beta}(\mathbb R^d)$. Assume that for all $t>0$ 
  \begin{equation}\label{hypo-3c} \lim_{j\to+\infty}\limsup_{R\to+\infty}{1\over R^d}\int_{B_{R}}\sup_{|\xi|\le t} |f(y,\xi)- g^j(y,\xi)|\dy
=0,
 \end{equation}
 and that each $g^j$ is homogenizable. Then $f$ is homogenizable and  $$f_{\rm hom}(\xi)=\lim_{j\to+\infty}g^j_{\rm hom}(\xi)$$ for every $\xi\in\MM$. 
 \end{theorem}
 
 In the particular case $g^j=g$ we have the following stability result as a corollary.
 
 \begin{corollary}[Stability of homogenizability]\label{stabhom} Let $f,g\in\FF_{\alpha,\beta}(\mathbb R^d)$. Assume that  
 for all $t>0$
 \begin{equation}\label{hypo-3b}\lim_{R\to+\infty}{1\over R^d}\int_{B_{R}}\sup_{|\xi|\le t} |f(y,\xi)- g(y,\xi)|\dy=0.
 \end{equation} Then $f$ is homogenizable if and only if $g$ is homogenizable, and in this case $g_{\rm hom}=f_{\rm hom}$.
 \end{corollary}

 The dependence of $g^j$ on $j$ in Theorem \ref{stabhom} is justified by some homogenization results for weakly almost-periodic functionals, which are obtained by approximation, as in the following example.
\begin{example}[Almost-periodic homogenization]\rm 
Theorem 4.1 can be used to obtain a homogenization theorem for quadratic forms $f(y,\xi)=a(y)|\xi|^2$ with $a$ almost periodic in the mean; that is, such that there exist trigonometric polynomials $P_j$ satisfying 
\begin{equation}\label{app-qp-coeff}
\lim_{j\to+\infty} \lim_{R\to+\infty}{1\over R^d} \int_{B_R} |P_j(y)- a(y)|\dy=0.
\end{equation}
Then we can define $g^j(y,\xi)= ((P_j(y)\vee\alpha)\wedge\beta)|\xi|^2$, whose coefficient is uniformly almost-periodic and hence homogenizable  (see e.g.~\cite{Koz1,B0}). By construction $g^j$ satisfy \eqref{hypo-3c}, so that we conclude the homogenizability of $f$.

More in general, Theorem 4.1 can be used to obtain a homogenization theorem for 
functions $f\in\FF_{\alpha,\beta}(\mathbb R^d)$ with the property that 
 for every $\xi\in\MM$ there exists a sequence of trigonometric polynomials $P^\xi_j$ such that
\begin{equation}\label{app-qp}
\lim_{j\to+\infty} \lim_{R\to+\infty}{1\over R^d} \int_{B_R} |P^\xi_j(y)- f(y,\xi)|\dy=0.
\end{equation}

Indeed, if such a condition holds we can construct a family $g^j$ such that \eqref{hypo-3c} holds and $g^j(\cdot,\xi)$ is almost periodic uniformly with respect to $\xi$ (that is, with almost periods that are independent of $\xi$) (see \cite{B1}), and hence homogenizable  (see \cite{B0}). For details on homogenization of uniformly almost-periodic functions and the construction of $g^j$ we refer to \cite[Chapters 15 and 17]{BDF}, respectively.
\end{example}

 \subsection{Stochastic homogenization}\label{stock}

 Let $({\OO},\mathcal T,P)$ be a probability space; let $\FF_{\rm stoc}$ be the collection of random integrands $f=f(\omega, x,\xi):\OO\times \mathbb R^d\times\mathbb M^{m\times d}\to\mathbb R$ satisfying the following properties:
 
 (i) $f$ is $\mathcal T\times \mathcal B(\mathbb R^d)\times \mathcal B(M^{m\times d})$-measurable;
 
 (ii) for all $\omega\in\OO$ we have $f(\omega,\cdot,\cdot)\in\FF_{\alpha,\beta}(\mathbb R^d)$.
 
 \smallskip
 
 We suppose that ${\OO}$ is equipped with a group of $P$-preserving transformations $\tau_z:{\OO}\to{\OO}$ labelled by $z\in \mathbb Z^d$.
 An integrand $f\in \FF_{\rm stoc}$ is said to be {\em stochastically periodic} if 
 $$
 f(\omega,x+z,\xi)=  f(\tau_z(\omega),x,\xi) \hbox{ for all }(\omega,z,x,\xi)\in {\OO}\times \mathbb Z^d\times\mathbb R^d\times\mathbb M^{m\times d}.
 $$
 The group $(\tau_z)_{z\in\mathbb Z^d}$ is said to be {\em ergodic} if every set $E\in \mathcal T$ such that $\tau_z(E)=E$ for every $z$ satisfies $P(E)=0$ or $P(E)=1$.

 If $f\in \FF_{\rm stoc}$  then we use the notation $f_\e=f(\omega,x/\e,\xi)$ and $F_\e(\omega,u,U)$ for the corresponding functional.
 The following stochastic homogenization results has been proved in \cite{DalMasoModica}.
 
 \begin{theorem}\label{DMM-th}
 If $f$ is stochastically periodic then for almost all $\omega$ there exists $f_{\rm hom}\in \FF_{\rm stoc}$, with $f_{\rm hom}=f_{\rm hom}(\omega,\xi)$ independent of $x$, such that $F_\e(\omega,\cdot,U)$ $\Gamma$-converge to $F_{\rm hom}(\omega,\cdot,U)$ corresponding to $f_{\rm hom}$ for every $U$. Furthermore, if $(\tau_z)_{z\in\mathbb Z^d}$ is also ergodic, then $f_{\rm hom}=f_{\rm hom}(\xi)$ is $P$-almost everywhere independent of $\omega$.
\end{theorem}
 
A first result for stochastic integrands is the following theorem, which can be obtained from Theorem {\rm\ref{DMM-th}} by applying Corollary \ref{stabhom} to $P$-almost all $\omega\in {\OO}$.

   \begin{theorem}Let $f, g\in \FF_{\rm stoc}$. Assume that $f$ is stochastically periodic,  and assume that
\begin{equation}\label{hypo-3c-stoc} \lim_{R\to+\infty}{1\over R^d}\int_{B_{R}}\sup_{|\xi|\le t} |f(\omega,y,\xi)- g(\omega, y,\xi)|\dy
=0
 \end{equation}
for all $t>0$ and for $P$-almost all $\omega\in{\OO}$. If $G_\e(\omega,\cdot,U)$ is the functional with integrand $g_\e=g(\omega,x/\e,\xi)$, then 
 $G_\e(\omega,\cdot,U)$ $\Gamma$-converge to   $F_{\rm hom}(\omega,\cdot,U)$ given by Theorem {\rm\ref{DMM-th}} for every $U\in\AU$ and $P$-almost all $\omega\in{\OO}$.
\end{theorem}

The pointwise condition \eqref{hypo-3c-stoc} can be replaced by the convergence to zero of the corresponding expectations. In this case the result is much weaker, since the set of probability zero that is excluded may depend on the   sequences $\e_k$ and $\e_{k_j}$ we are considering.  

\begin{theorem}\label{stockholm}  Let $f, g\in \FF_{\rm stoc}$. Assume that $f$ is stochastically periodic,  and assume that
 \begin{equation}\label{hypo-3b-stoc} \lim_{R\to+\infty}{1\over R^d}\int_{\OO}
 \bigg(\int_{B_{R}}\sup_{|\xi|\le t} |f(\omega,y,\xi)- g(\omega, y,\xi)|\dy\bigg)\,d P(\omega)
=0
 \end{equation}
 for all $t>0$.  Then for every sequence  $\e_k\to0^+$ there exists a subsequence  $\e_{k_j}$ such that $G_{\e_{k_j}}(\omega,\cdot,U)$   $\Gamma$-converges   to $F_{\rm hom}(\omega,\cdot,U)$   for $P$-almost all $\omega\in{\OO}$ and for every $U\in\AU$ .
 \end{theorem}

\begin{proof} We fix a sequence $\e_k\to0^+$. 
In analogy with the notation in the previous sections, we set
$$ \psi(\omega,y,\xi):=|f(\omega,y,\xi)- g(\omega, y,\xi)|.
$$
With fixed $  n\in \N$, we set $R_k=  n /\e_k$, and observe that \eqref{hypo-3b-stoc} implies  that  
  \begin{eqnarray}\label{hypo-71}
\lim_{k\to+\infty}{1\over R_k^d}\int_{\OO}\bigg(\int_{  B_{R_k}}\sup_{|\xi|\le t} \psi(\omega,y,\xi)\dy \bigg)\,dP(\omega)
=0 .
\end{eqnarray}
By the change of variable $y=\frac{x}{\e_k}$ we then also have 
  \begin{eqnarray}\label{hypo-72}
\lim_{k\to+\infty}{1\over   n^d}\int_{\OO}\bigg(\int_{  B_{n}}\sup_{|\xi|\le t} \psi\Big(\omega,\frac{x}{\e_k},\xi\Big)\dx\bigg)\,dP(\omega)
=0.
\end{eqnarray}
Hence, there exists a subsequence $\e_{k_j}$ such that 
  \begin{eqnarray}\label{hypo-73}
\lim_{j\to+\infty}\int_{  B_{n}}\sup_{|\xi|\le t} \psi\Big(\omega,\frac{x}{\e_{k_j}},\xi\Big)\dx
=0
\end{eqnarray}
for $P$-almost every $\omega\in{\OO}$. 
  By a diagonal argument   we can suppose that the subsequence $\e_{k_j}$ does not depend on $  n\in\N$. 
We now fix $\omega\in {\OO}$ such that \eqref{hypo-73} holds for every $  n\in\N$. 
  For every $x_0\in \R^d$ and every $\rho>0$ there exists $n\in\N$ such that $B_\rho(x_0)\subset B_n$. By 
 \eqref{hypo-73} this implies that   
$$
\lim_{j\to+\infty}{1\over \rho^d}\int_{B_{\rho}(x_0)}\sup_{|\xi|\le t} \psi\Big(\omega,\frac{x}{\e_{k_j}},\xi\Big)\dx=0.$$
By Corollary \ref{Corgen} we then deduce that $G_{\e_{k_j}}(\omega,\cdot,U)$ $\Gamma$-converge to $F_{\rm hom}(\omega,\cdot,U)$.
\end{proof}

\section{$G$-convergence}\label{G-conv}
The results of the previous sections can be translated into results about the $G$-convergence of elliptic operators according to the following definition. A more direct approach, but based on the same ideas, will be described in the next section. For simplicity of exposition in this and the next sections we treat only problems for scalar-valued functions, but all results remain valid, with obvious modifications, for vector-valued $u$. 

\smallskip
We assume that $\DD$ is an open set of $\mathbb R^d$, and $ {\mathbb M}_{\rm sym}^{d\times d}$  denotes the space of symmetric $d\times d$ matrices.
Given $\alpha,\beta>0$ with $\alpha\le \beta$,  we let $\mathcal M^{\rm sym}=\mathcal M^{\rm sym}_{\alpha,\beta}(\DD)$ denote the family of all measurable $A\colon\DD\to {\mathbb M}_{\rm sym}^{d\times d}$, such that 
\begin{equation}\label{stime-G}
\alpha |\xi|^2\le \langle A(x)\xi,\xi\rangle\le \beta|\xi|^2\qquad\hbox{ for almost all $x\in\DD$ and all $\xi\in\mathbb R^d$.}
\end{equation} 

\begin{definition} Let $D$ be a bounded open set in $\mathbb R^d$, let $A_k$ be a sequence in $\mathcal M^{\rm sym}$, and let $A\in\mathcal M^{\rm sym}$. We say that $A_k$ {\em G-converges} to $A$ if for all $\phi\in H^{-1}(\DD)$ the sequence of solutions $u_k$ in the sense of distribution of 
$$
\begin{cases} -{\rm div}(A_k\nabla u_k)=\phi \\
u_k\in H^1_0(\DD)
\end{cases}
$$
converge weakly in $H^1_0(\DD)$ to the solution $u$ of 
$$
\begin{cases} -{\rm div}(A\nabla u)=\phi \\
u\in H^1_0(\DD).
\end{cases}
$$
\end{definition}
For the connection between $G$-convergence and spectral convergence of the elliptic operators we refer to \cite[Teorema 4.1]{BoMa} (see also \cite[Section 3.9.1]{Attouch}).
The link between $G$-convergence and $\Gamma$-convergence is given by the following theorem (see \cite[Section 22]{DM}).

\begin{theorem}\label{equiv-G-Gamma} Let $D$ be a bounded open set in $\mathbb R^d$, let $A_k, A\in \mathcal M^{\rm sym}$, and for every $U\in\AU$ let $F_k(\cdot, U), F(\cdot, U)$ be the quadratic functionals defined in $H^1(D)$ by
$$
F_k(u,U)=\int_U\langle A_k(x)\nabla u,\nabla u\rangle\dx,\qquad F(u,U)=\int_U\langle A(x)\nabla u,\nabla u\rangle\dx.
$$
Then the following conditions are equivalent.

{\rm(a)} $A_k$ G-converges to $A$;

{\rm(b)} $F_k(\cdot, U)$ $\Gamma$-converges to $F(\cdot, U)$ for every open set $U\in\AU$. 

\end{theorem}

From this equivalence, Theorem \ref{main} translates into the following result, after remarking that  for $B^j_\infty,B_\infty\in {\mathbb M}_{\rm sym}^{d\times d}$
the convergence of $\langle B^j_\infty\xi,\xi\rangle$ to $\langle B_\infty\xi,\xi\rangle$ for all $\xi$ is equivalent to the convergence of $B^j_\infty$ to $B_\infty$ by the polarization identity.

  \begin{theorem}\label{main-G}  Let $D$ be a bounded open set in $\mathbb R^d$ and
let $A_k, B^j_k$ be matrices in $\mathcal M^{\rm sym}$.
Assume that  for almost all $x\in\DD$ we have
\begin{equation}\label{hypo-z-G}
\lim_{j\to+\infty}\limsup_{\rho\to0}\limsup_{k\to+\infty}{1\over\rho^d}\int_{B_\rho(x)} |A_k(y)-B^j_k(y)|\dy=0.
\end{equation}
If  $B^j_k$ G-converge to $B^j_\infty$ then there exists the pointwise limit $B_\infty$ of $B^j_\infty$ and
$A_k$ G-converges to $B_\infty$.
 \end{theorem}

As for Corollary \ref{Corgen}, from Theorem \ref{main-G} we obtain the following result.
\begin{corollary}\label{Corgen-G}  Let $D$ be a bounded open set in $\mathbb R^d$,
and let $A_k, B_k$ be two matrix-valued sequences in $\mathcal M^{\rm sym}$. Assume that $A_k$ G-converges to $A$. 
Assume also that for almost all $x\in\DD$ we have
\begin{equation}\label{hypo-G}
\lim_{{ \rho\to0}}\limsup_{k\to+\infty}{1\over\rho^d}\int_{B_\rho(x)} |A_k(y)-B_k(y)|\dy=0.
\end{equation}
Then $B_k$ G-converges to $A$.
 \end{corollary}

\begin{remark} \rm
Note that condition \eqref{hypo-G} is satisfied if $ A_k(y)-B_k(y)\to 0$ for almost all $y\in\DD$, since we can apply the Dominated Convergence Theorem thanks to \eqref{stime-G}.
\end{remark}

 \begin{remark} \rm
 If $A_k$ and $B_k$ satisfy \eqref{hypo-G}, for given $\phi\in L^2(\DD)$, let $u_k, v_k$ be the solutions to
 $$
\begin{cases} -{\rm div}(A_k\nabla u_k)=\phi \\
u_k\in H^1_0(\DD)
\end{cases}\qquad 
\begin{cases} -{\rm div}(B_k\nabla v_k)=\phi \\
v_k\in H^1_0(\DD).
\end{cases}
$$
Then the difference $u_k-v_k$ tends to $0$ weakly in $H^1_0(\DD)$. By the compactness of $G$-con\-ver\-gence, 
which we can deduce from Theorems \ref{equiv-G-Gamma} and Theorem \ref{comp}, passing to a subsequence we can assume that $A_k$ G-converges to some $A$, and apply Corollary \ref{Corgen-G} and the definition of $G$-convergence.
We then obtain that $u_k$ and $v_k$ converge to the same function weakly in $H^1_0(\DD)$; hence, $u_k-v_k$ tends to $0$ weakly in $H^1_0(\DD)$. Since the limit does not depend on the subsequence we deduce the convergence of  the full sequence. 
\end{remark}
 
\begin{remark}\label{re-nec}\rm In the one-dimensional case, $G$-convergence of $a_k$ to $a$ is equivalent to the weak$^*$ convergence  of ${1\over a_k}$ to ${1\over a_0}$ in $L^\infty(\DD)$. 
Hence, we have that the necessary and sufficient condition for $a_k$ and $b_k$ to have the same G-limit is that $1/a_k$  weakly$^*$  converges in $L^\infty(\DD)$, and
  \begin{equation}\label{hypo-7}
{1\over a_k}-{1\over b_k}\wto 0\ \hbox{ weakly$^*$ in } L^\infty(\DD).
\end{equation}
We note that condition \eqref{hypo}, which translates into
\begin{equation}\label{hypo-8}
\lim_{\rho\to0}\limsup_{k\to+\infty}{1\over\rho}\int_{B_\rho(x)} |a_k(y)-b_k(y)|\dy=0,
\end{equation}
implies \eqref{hypo-7}. Indeed, let $\theta$ be the weak$^*$ limit of ${1\over a_k}-{1\over b_k}$. For almost all $x\in\DD$, we have
\begin{eqnarray*}|\theta(x)|&=&\bigg|\lim_{\rho\to 0}\lim_{k\to+\infty} {1\over 2\rho}\int_{B_\rho(x)}\Big({1\over a_k(y)}-{1\over b_k(y)}\Big)dy\bigg|\\
&= &\lim_{\rho\to 0}\lim_{k\to+\infty} \bigg|{1\over 2\rho}\int_{B_\rho(x)}\Big({b_k(y)-a_k(y)\over a_k(y) b_k(y)}\Big)dy\bigg|
\\
&
\le & \lim_{\rho\to 0}\lim_{k\to+\infty} {1\over 2\alpha^2\rho}\int_{B_\rho(x)}|b_k(y)-a_k(y)|\dy =0
\end{eqnarray*}
by \eqref{hypo-8}. Hence $\theta(x)=0$ for almost all $x\in\DD$. 

Conversely, note that \eqref{hypo-G} is not necessary. Indeed, take $a$ and $b$ two $1$-periodic functions taking values in $[\alpha,\beta]$ such that 
$$\int_0^1{1\over a(y)}\,dy= \int_0^1{1\over b(y)}\,dy\quad\hbox{ but  } a(y)\neq b(y) \hbox{ on a set of non-zero measure.}$$
If $a_k(x)= a(kx)$ and $b_k(x)= b(kx)$ then \eqref{hypo-7} holds, but, by the weak$^*$-limit of $|a_k-b_k|$ to the constant 
$\int_0^1|a(y)-b(y)|dy$, we have
$$
\lim_{\rho\to0}\lim_{k\to+\infty}{1\over\rho}\int_{B_\rho(x)} |a_k(y)-b_k(y)|\dy=2\int_0^1|a(y)-b(y)|\dy\neq0,
$$
and \eqref{hypo-8} does not hold. 

For non-quadratic energies in dimension one similar arguments can be used to obtain Corollary \ref{Corgen} remarking that $\Gamma$-conver\-gence is characterized as the weak convergence of the Legendre transforms of the integrands (Theorem 2.35 in \cite{GCB}). 
\end{remark}

\begin{remark}\rm
In dimension $d\ge 2$, we consider the case of layered materials, for which $A_k(x)= a(kx_1) I$, where $I$ is the identity matrix and $a$ is a $1$ periodic function on $\mathbb R$ as in the previous example. In this case, the homogenized matrix  $A_{\rm hom}$ is diagonal, whose $11$-entry is the harmonic mean of $a$, while the other diagonal entries are all equal to the average of $a$.
Hence, if $a$ and $b$ are periodic functions with the same harmonic means and averages, they produce the same homogenized matrix, although in general hypothesis \eqref{hypo-8} does not hold. 

A trivial case is when $a(t)=\alpha $ for $0\le t<\frac12$ and $a(t)=\beta $ for $\frac12\le t<1$, while $b(t)=\beta $ for $0\le t<\frac12$ and $b(t)=\alpha $ for $\frac12\le t<1$, so that $|a(t)-b(t)|=\beta-\alpha$ for all $t$.
\end{remark}

We now consider homogenization problems in the context of $G$-convergence, when the matrix-valued $A_k\in\mathcal M^{\rm sym}_{\alpha,\beta}(\DD)$ can be obtained by scaling a single $A$ defined on the whole space; namely, there exists $A\in\mathcal M^{\rm sym}_{\alpha,\beta}(\mathbb R^d)$ and $\e_k$, a positive sequence tending to $0$, such that
 \begin{equation}
 A_k(x)= A\Big({x\over\e_k}\Big).
 \end{equation}
We say that $A$ is {\em homogenizable},  if for all bounded open sets $D$ the sequence $A_k$ G-converges to some $A_{\rm hom}$ independent of $x$ and of the sequence $\e_k$.

Theorem \ref{th-hom-j} translates into the following result.

  \begin{theorem}[Homogenizability by approximation]
  Let $A\in\mathcal M^{\rm sym}_{\alpha,\beta}(\mathbb R^d)$, and let $B^j$ be a sequence in $\mathcal M^{\rm sym}_{\alpha,\beta}(\mathbb R^d)$. Assume that \begin{equation}\label{hypo-3c-G} \lim_{j\to+\infty}\limsup_{R\to+\infty}{1\over R^d}\int_{B_{R}} |A(x)- B^j(x)|\dx
=0,
 \end{equation}and that each $B^j$ is homogenizable. Then $A$ is homogenizable and  $A_{\rm hom}=\lim\limits_{j\to+\infty}B^j_{\rm hom}$. 
 \end{theorem}

 In the particular case $B^j=B$ we have the following stability result as a corollary.
 
 \begin{corollary}[Stability of homogenizability]\label{stabhom-G} Let $A,B\in\mathcal M^{\rm sym}_{\alpha,\beta}(\mathbb R^d)$. Assume that  \begin{equation}\label{hypo-3b-G}\lim_{R\to+\infty}{1\over R^d}\int_{B_{R}} |A(x)- B(x)|\dx=0.
 \end{equation} Then $A$ is homogenizable if and only if $B$ is homogenizable, and in this case $A_{\rm hom}=B_{\rm hom}$.
 \end{corollary}

  Note that \eqref{hypo-3c-G} and \eqref{hypo-3b-G} are satisfied if $A- B^j$ and $A-B$, respectively, belongs to $L^p(\mathbb R^d)$ for some $p\in[1,+\infty)$. 

 \begin{remark}\rm As noted in Remark \ref{re-nec}, condition \eqref{hypo-3b-G} is not necessary to the validity of the claim. 
 Again resorting to the one-dimensional case, in which we have functions $a,b$ in the place of matrices $A,B$, for which condition \eqref{hypo-3b-G} reads
  \begin{equation}\label{hypo-3bbc}\lim_{R\to+\infty}{1\over R}\int_{-R}^R|a(y)- b(y)|\dx=0,
 \end{equation} 
 we also see that the weaker condition
 \begin{equation}\label{hypo-3bb}\lim_{R\to+\infty}{1\over R}\int_{-R}^R(a(y)- b(y))\dx=0
 \end{equation} 
 does not allow to deduce the homogenizability of $a$ from that of $b$.
 To check this, it suffices to take $b$ identically equal to a positive constant $\gamma\in(\alpha,\beta)$, and 
 $$
 a(x)=\begin{cases}\gamma+c &\hbox{ if $x\ge0$}\\
 \gamma-c&\hbox{ if $x<0$,}\end{cases}
 $$
 with $c$ a constant such that $\gamma\pm c\in(\alpha,\beta)$. Then \eqref{hypo-3bb} holds but $a$ is not homogenizable. Indeed, $a_k=a$  do not depend on $k$, so that they G-converge to $a$.  \end{remark}

Finally, we specialize $G$-convergence to the stochastic case.
 Let $({\OO},\mathcal T,P)$ be a probability space; let ${\mathcal M}_{\rm stoc}$ be the collection of random matrices $A=A(\omega, x):\OO\times \mathbb R^d\to{\mathbb M}_{\rm sym}^{d\times d}$ satisfying the following properties:
 
 (i) $A$ is $\mathcal T\times \mathcal B(\mathbb R^d)$-measurable;
 
 (ii) for all $\omega\in\OO$ we have $A(\omega,\cdot)\in\mathcal M^{\rm sym}_{\alpha,\beta}(\mathbb R^d)$.

Given  a group of $P$-preserving transformations the notion of stochastic periodicity is obtained by modifying the definition in Section {\rm\ref{stock}. If $A\in\mathcal M_{\rm stoc}$ is stochastically periodic, then from \cite{Koz2,JKO,DalMasoModica} we have that for almost all $\omega$ the matrix $A(\omega,\cdot)$ is homogenizable.  The following result is an immediate consequence of Theorem \ref{stockholm}.

\begin{theorem}
Let $A, B\in\mathcal M_{\rm stoc}$, with $A$ stochastically periodic, and for $P$-almost every $\omega\in\OO$ let $A_{\rm hom}(\omega)$ be its homogenized matrix.  Assume in addition that
\begin{equation}\label{hypo-3c-stoc-G} \lim_{R\to+\infty}{1\over R^d}\int_{\OO}\Big(\int_{B_{R}}|A(\omega,x)- B(\omega, x)|\dx\Big)dP(\omega)
=0.
 \end{equation}
Then, given a sequence $\e_k$ of positive numbers converging to $0$  there exists a subsequence $\e_{k_j}$ such that  $B_{j}(\omega,x)= B\big(\omega,{x\over\e_{k_j}}\big)$   $G$-converges to  $A_{\rm hom}(\omega)$ almost surely.   
\end{theorem}

\section{$H$-convergence}
Non-symmetric elliptic operators fall in the range of application of the stability results by $\Gamma$-convergence  only  upon significantly modifying the classical setting in Section \ref{G-conv} and using differential constraints (see \cite{ADM} for details).
In this section we follow a more direct approach, which avoids the use of $\Gamma$-convergence. 
Although we cannot exploit the results of Sections \ref{pre-sec} and \ref{closure-sec}, 
we follow the same ideas based on compactness and localisation arguments.

\smallskip
Since we may also consider non-symmetric matrices, for $\alpha,\beta>0$ with $\alpha\le \beta$, and $\DD$ an open set in $\mathbb R^d$  we let $\mathcal M=\mathcal M_{\alpha,\beta}(\DD)$ denote the family of all measurable $A\colon\DD\to {\mathbb M}^{d\times d}$, such that 
\begin{equation}\label{stime-H} 
\alpha |\xi|^2\le \langle A(x)\xi,\xi\rangle \ \hbox{ and }\quad{ |\xi|^2\le \beta \langle A^{-1}(x)\xi,\xi\rangle}\ \hbox{ for almost all $x\in\DD$ and $\xi\in\mathbb R^d$.}
\end{equation}

\begin{definition} Let $\DD$ be a bounded open set, $A_k$ be a sequence in $\mathcal M_{\alpha,\beta}(\DD)$. We say that $A_k$ {\em H-converges} to $A\in\mathcal M_{\alpha,\beta}(\DD)$ if for all $\phi\in H^{-1}(\DD)$ the sequence of solutions $u_k$ in the sense of distribution of 
$$
\begin{cases} -{\rm div}(A_k\nabla u_k)=\phi \\
u_k\in H^1_0(\DD)
\end{cases}
$$
converges weakly in $H^1_0(\DD)$ to the solution $u$ of 
$$
\begin{cases} -{\rm div}(A\nabla u)=\phi \\
u\in H^1_0(\DD)\, ,
\end{cases}
$$
and we also have
$
A_k\nabla u_k \wto   A\nabla u \hbox{ weakly in } L^2(\DD;\mathbb R^d)$.
\end{definition}

\begin{remark}[Localization of $H$-convergence]\label{loc-H}\rm  Let 
$A_k$ H-converge to $A$.
It is proven in \cite{MT} that if $\DD'$ is an open subset of $\DD$ and $w\in H^1(\DD')$ then  the sequence of solutions $u_k$ in the sense of distribution of 
$$
\begin{cases} -{\rm div}(A_k\nabla u_k)=\phi \\
u_k-w\in H^1_0(\DD')
\end{cases}
$$
converges weakly in $H^1(\DD')$ to the solution $u$ of 
$$
\begin{cases} -{\rm div}(A\nabla u)=\phi \\
u-w\in H^1_0(\DD')\,,
\end{cases}
$$
and we also have $A_k\nabla u_k \wto   A\nabla u \hbox{ weakly in } L^2(\DD';\mathbb R^d)$.
\end{remark}

The corresponding result to Theorem \ref{derit} is the following derivation formula, where now $\ell_\xi(x)=\langle \xi,x\rangle$ denotes the scalar linear function with gradient $\xi$.

\begin{theorem}
Let $A\in\mathcal M$, and let $x_0$ be a Lebesgue point of $A$, and let $\xi\in\mathbb R^d$.  For every $\rho>0$ consider the problem
\begin{equation}\label{HA}
\begin{cases} -{\rm div}(A\nabla u_\rho)=0 \\
u_\rho-\ell_\xi\in H^1_0(B_\rho(x_0)).\end{cases}
\end{equation}
Then we have
\begin{equation}\label{0-moment}
A(x_0)\xi=\lim_{\rho\to 0} \frac1{|B_\rho|}\int_{B_\rho(x_0)} A\nabla u_\rho\dx.
\end{equation}
\end{theorem}

\begin{proof} 

For notational simplicity, we can assume that $x_0=0$. For $y\in B_1$ we define $A_\rho(y)= A(\rho y)$, and $v_\rho(y)= \frac1\rho u_\rho(\rho y)$,
so that equation \eqref{HA} is rewritten as 
 \begin{equation}\label{HA-v}
\begin{cases} -{\rm div}(A_\rho\nabla v_\rho)=0 \\
v_\rho-\ell_\xi\in H^1_0(B_1).
\end{cases}
\end{equation}
Since $0$ is a Lebesque point for $A$, we have  $A_\rho\to A(0)$  in $L^1(B_1)$, so that, by \cite{MT} we also have the corresponding $H$-convergence, which implies that 
$A_\rho\nabla v_\rho$ tends to $A(0)\nabla v$, where $v$ is the unique solution to
 \begin{equation}\label{HA-0}
\begin{cases} -{\rm div}(A(0)\nabla v)=0 \\
v-\ell_\xi\in H^1_0(B_1)\,,
\end{cases}
\end{equation}
which coincides with $\ell_\xi$. This implies \eqref{0-moment} after a change of variables. 
\end{proof}

The analog of Corollary \ref{Corgen} is the following result.

\begin{theorem}[Stability of $H$-convergence]\label{Corgen-H}
Let $A_k, B_k$ be two sequences of matrices in $\mathcal M$. Assume that $A_k$ H-converges to $A$ in $\DD$.
 Assume also that   for almost all $x\in\DD$ we have
\begin{equation}\label{hypo-H}
\lim_{{ \rho\to0}}\limsup_{k\to+\infty}{1\over\rho^d}\int_{B_\rho(x)}|A_k-B_k|\dy=0.
\end{equation}
Then also $B_k$ H-converges to $A$. Furthermore, if $A,B\in\mathcal M_{\alpha,\beta}(\mathbb R^d)$ and  \eqref{hypo-3b-G} holds, then $A$ is homogenizable if and only if $B$ is homogenizable, and in this case $A_{\rm hom}=B_{\rm hom}$.
 \end{theorem}
 
 \begin{proof} By the compactness of $H$-convergence (see \cite{MT}) we may assume that $B_k$ $H$-converges to $B$, so that we only have to prove that $B=A$ almost everywhere. We fix $x_0$ a Lebesgue point for both $A$ and $B$. For notational simplicity, we can assume that $x_0=0$.
 
Let $\xi\in\mathbb R^d$ and consider the solutions $u_k^\rho$ and $v_k^\rho$ to 
\begin{equation}\label{hypo-H2}
\begin{cases} -{\rm div}(A_k\nabla u_k^\rho)=0 \\
u_k^\rho-\ell_\xi\in H^1_0(B_\rho),
\end{cases}\qquad
\begin{cases} -{\rm div}(B_k\nabla v_k^\rho)=0 \\
v_k^\rho-\ell_\xi\in H^1_0(B_\rho),
\end{cases}
\end{equation}
 respectively. Moreover, we also consider the solutions $u^\rho$ and $v^\rho$ to 
\begin{equation}\label{hypo-H3}
\begin{cases} -{\rm div}(A\nabla u^\rho)=0 \\
u^\rho-\ell_\xi\in H^1_0(B_\rho),
\end{cases}\qquad
\begin{cases} -{\rm div}(B\nabla v^\rho)=0 \\
v^\rho-\ell_\xi\in H^1_0(B_\rho),
\end{cases}
\end{equation}
respectively. By Remark \ref{loc-H} we have
\begin{eqnarray}\label{hypo-H4}
\lim_{k\to+\infty}\int_{B_\rho} (A_k\nabla u_k^\rho-B_k\nabla v_k^\rho)\dy=\int_{B_\rho} (A\nabla u^\rho-B\nabla v^\rho)\dy.
\end{eqnarray}
In order to show that $A(0)=B(0)$, by \eqref{0-moment} and the arbitrariness of $\xi$ it is then equivalent to show that
 \begin{eqnarray}\label{hypo-H5}
\lim_{\rho\to 0}\lim_{k\to+\infty}\frac1{|B_\rho|}\int_{B_\rho} (A_k\nabla u_k^\rho-B_k\nabla v_k^\rho)\dy=0.
\end{eqnarray}
To that end, we write
 \begin{eqnarray}\label{hypo-H6} \nonumber
&&\Big|\frac1{|B_\rho|}\int_{B_\rho} (A_k\nabla u_k^\rho-B_k\nabla v_k^\rho)\dy\Big|\\  \nonumber
&&=\Big|\frac1{|B_\rho|}\int_{B_\rho} ((A_k-B_k)\nabla u_k^\rho+B_k (\nabla u_k^\rho-\nabla v_k^\rho))\dy\Big|\\  \nonumber
&&\le\frac1{|B_\rho|}\int_{B_\rho} |A_k-B_k||\nabla u_k^\rho|\dy+\frac1{|B_\rho|}\int_{B_\rho} |B_k||\nabla u_k^\rho-\nabla v_k^\rho|\dy\\  \nonumber
&&\le\Big(\frac1{|B_\rho|}\int_{B_\rho} |A_k-B_k|^2\dy\Big)^{\frac12}\Big(\frac1{|B_\rho|}\int_{B_\rho}|\nabla u_k^\rho|^2\dy\Big)^{\frac12}\\  
&& \qquad +\Big(\frac1{|B_\rho|}\int_{B_\rho} |B_k|^2\dy\Big)^{\frac12}\Big(\frac1{|B_\rho|}\int_{B_\rho}|\nabla u_k^\rho-\nabla v_k^\rho|^2\dy\Big)^{\frac12}.
\end{eqnarray}
Note that 
$$
\lim_{\rho\to 0}\lim_{k\to+\infty}\frac1{|B_\rho|}\int_{B_\rho} |A_k-B_k|^2\dy=0\qquad \hbox{ and }\qquad\frac1{|B_\rho|}\int_{B_\rho} |B_k|^2\dy\le \beta^2
$$
by \eqref{hypo-H} and \eqref{stime-H}. 
By multiplying the left-hand equation in \eqref{hypo-H2} by $u^\rho_k-\ell_\xi$ and integrating on $B_\rho$ we obtain 
$$
\frac1{|B_\rho|}\int_{B_\rho}|\nabla u_k^\rho|^2\dy\le \frac{\beta^2}{\alpha^2}|\xi|^2
$$
for all $k$ and $\rho$. In order to estimate the last term in \eqref{hypo-H6}
we note that 
\begin{equation}\label{hypo-H7}
\begin{cases} -{\rm div}(A_k(\nabla u_k^\rho-\nabla v_k^\rho))={\rm div}((A_k-B_k)\nabla v_k^\rho)) \\
u_k^\rho-v_k^\rho\in H^1_0(B_\rho)
\end{cases}
\end{equation}
by \eqref{hypo-H2}.
Multiplying the equation in \eqref{hypo-H7} by $u_k^\rho-v_k^\rho$ and integrating on $B_\rho$ we get, if $\frac1p+\frac1q=\frac12$, 
\begin{eqnarray*}\label{hypo-H8}
&&\hskip-.5cm\frac\alpha{|B_\rho|}\int_{B_\rho}|\nabla u_k^\rho-\nabla v_k^\rho|^2\dy
\le \frac1{|B_\rho|}\int_{B_\rho}|A_k-B_k||\nabla v_k^\rho||\nabla u_k^\rho-\nabla v_k^\rho|\dy\\ 
&&\le
\Big(\frac1{|B_\rho|}\int_{B_\rho}|A_k-B_k|^p\dy\Big)^{\frac1p}
\Big(\frac1{|B_\rho|}\int_{B_\rho}|\nabla v_k^\rho|^q\dy\Big)^{\frac1q}
\Big(\frac1{|B_\rho|}\int_{B_\rho}|\nabla u_k^\rho-\nabla v_k^\rho|^2\dy\Big)^{\frac12},
\end{eqnarray*}
so that
\[
\Big(\frac1{|B_\rho|}\int_{B_\rho}|\nabla u_k^\rho-\nabla v_k^\rho|^2\dy\Big)^{\frac12}
 \le \frac1\alpha
\Big(\frac1{|B_\rho|}\int_{B_\rho}|A_k-B_k|^p\dy\Big)^{\frac1p}
\Big(\frac1{|B_\rho|}\int_{B_\rho}|\nabla v_k^\rho|^q\dy\Big)^{\frac1q}.
\]
 In order to show that this last right-hand side tends to $0$ as $k\to+\infty$ it suffices to show that there exists $q>2$ such that the last term is bounded uniformly with respect to $\rho$ and $k$. To this end, we define the scaled functions $z^\rho_k(x)= \frac1\rho v_k^\rho(\rho x)$ and matrices $B^\rho_k(x)=  B_k(\rho x)$, and note that they satisfy 
 \begin{equation}\label{HA-w}
\begin{cases} -{\rm div}(B^\rho_k\nabla z^\rho_k)=0 \\
z^\rho_k-\ell_\xi\in H^1_0(B_1).
\end{cases}
\end{equation}
By the Meyers-Elcrat higher-integrability theorem \cite[Theorem 2]{M-E}
(see also \cite{gallouet}), there exists $q>2$, and $C>0$, independent of $\rho$ and $k$, such that 
$$
\frac1{|B_\rho|}\int_{B_\rho}|\nabla v_k^\rho|^q\dy=\frac1{|B_1|}\int_{B_1}|\nabla z_k^\rho|^q\dy\le C,
$$
which proves the claim.
\end{proof}

Let $A\in\mathcal M_{\alpha,\beta}(\mathbb R^d)$, let $\e_k$ be a positive sequence tending to $0$, and let
 \begin{equation}\label{ak-hom}
 A_k(x)= A\Big({x\over\e_k}\Big).
 \end{equation}
We say that $A$ is {\em homogenizable} if $A_k$ H-converges to some $A_{\rm hom}$ independent of $x$ and of the sequence $\e_k$.

\begin{corollary}
If  $A,B\in\mathcal M_{\alpha,\beta}(\mathbb R^d)$ and  \eqref{hypo-3b-G} holds, then $A$ is homogenizable if and only if $B$ is homogenizable, and in this case $A_{\rm hom}=B_{\rm hom}$.
\end{corollary}

\begin{proof}
It suffices to note that \eqref{hypo-3b-G} guarantees the validity of \eqref{hypo-H} thanks to the change of variables as in \eqref{hypo-4} and \eqref{hypo-3}.
\end{proof}

Finally, we note that the same remarks as for $G$-convergence at the end of Section \ref{G-conv} also hold for the stochastic homogenization by $H$-convergence.

\section{Perforated domains with Neumann boundary conditions}
We now show how we can derive a stability result for the homogenization of minimum problems in perforated domains. We only treat the quadratic case in order to concentrate on the role of the perforations.
\subsection{General assumptions}
We consider closed sets $E\subset \mathbb R^d$ such that there exists an extension operator $T: H^1(\mathbb R^d\setminus E)\to H^1(\mathbb R^d)$ satisfying the following properties
\begin{equation}\label{id-u}
Tu(x)= u(x) \hbox{ for almost all } x\in\mathbb R^d\setminus E,
\end{equation}
\begin{equation}\label{extDu}
\|\nabla Tu\|_{L^2(\mathbb R^d)}\le C \|\nabla u\|_{L^2(\mathbb R^d\setminus E)}
\end{equation}
for some constant $C>0$. Note that \eqref{extDu} implies the connectedness of $\mathbb R^d\setminus E$. 
In the case of bounded domains $\DD$, the definition of extension operator is more technically complex since it must take into account that $\partial \DD$ may disconnect the set $E$; we refer to \cite{ACPDMP} for details.

Properties \eqref{id-u} and \eqref{extDu} are satisfied if $\mathbb R^d\setminus E$ is connected and with Lipschitz boundary \cite{ACPDMP}. A non-periodic example is given by sets $E=\sum_{i} E_i$ such that $E_i\subset B_{r_i}(x_i)$ with $B_{r_i}(x_i)\cap B_{r_j}(x_j)=\emptyset$ and such that there exists a constant $C>0$ and extension operators $T_i: H^1(B_{r_i}(x_i)\setminus E_i)\to H^1(B_{r_i}(x_i))$ such that $T_iu(x)= u(x)$ for almost all $x\in B_{r_i}(x_i)\setminus E_i$ and 
$\|\nabla T_iu\|_{L^2(B_{r_i}(x_i))}\le C \|\nabla u\|_{L^2(B_{r_i}(x_i)\setminus E_i)}$.
The simplest situation is when $E_i$ is a smaller ball concentric with $B_{r_i}(x_i)$.
 To check this it suffices to construct an extension operator when $E$ is the ball of radius $2$ and centre $0$.  We then define
$$
T(u)(x)=\begin{cases} \overline u & \hbox{ if }|x|<1\\
(|x|-1)u(L(x)) +(2-|x|) \overline u & \hbox{ if }1\le |x|<2 \\
u(x) & \hbox{ if }|x|>2,
\end{cases}
$$
where $L(x)= (4-|x|){x\over|x|}$ for $1<|x|<2$ and
$$
\overline u = {1\over |B_3\setminus B_2|}\int_{B_3\setminus B_2} u(x)\,dx.
$$
Then 
$$
\int_{B_2}|\nabla Tu|^2\dx\le C'\int_{B_3\setminus B_2}(|\nabla u|^2+ |u-\overline u|^2)\dx,
$$
and \eqref{extDu} follows by the Poincar\'e inequality. Since the constant $C$ in \eqref{extDu} for $E=B_2$ is invariant by homothety, we can define an extension operator separately for each ball maintaining the same constant.

Note that if $E$ satisfies \eqref{id-u} and \eqref{extDu}, then for all $\e>0$ there exists an extension operator $T_{\e}$ for $\e E$ satisfying  \eqref{id-u} and \eqref{extDu} with the same constant $C$ independent of $\e$.

In order to avoid trivial limits we suppose that
 \begin{equation}\label{dens-E}
\liminf_{R\to+\infty}{|B_R\setminus E|\over R^d}>0.
\end{equation}
 This condition implies that the weak limits $\psi$ of characteristic functions $1_{\mathbb R^d\setminus \e E}$ satisfy $\psi>0$ almost everywhere.

\subsection{Stability for perforated domains}
We consider functionals
\begin{equation}\label{quadFe}
F^E_\e(u)= \int_{\mathbb R^d\setminus \e E} |\nabla u|^2\dx
\end{equation} 
defined on $H^1(\mathbb R^d)$.  
Our goal is to compare the asymptotic behaviour of such functionals for different $E$ in the spirit of the stability results. The corresponding $\Gamma$-limits will be computed in the space $H^1(\mathbb R^d)$ endowed with the $L^2_{\rm loc}$-topology. In the following section we will examine the relation between $\Gamma$-convergence and the solutions to some minimum problems on perforated domains.

\begin{theorem}[Perturbations of perforated domains]\label{Perfdo}
 Let $E$ and $E'$ be two closed sets satisfying the extension properties 
\eqref{id-u} and \eqref{extDu}, and such that 
\begin{equation}\label{ci-ci}
\lim_{R\to+\infty} {1\over R^d} |(E\triangle E')\cap B_R|=0,
\end{equation}
where $E\triangle E':=(E\setminus E')\cup (E'\setminus E)$ denotes the symmetric difference between the sets.
Suppose that $E$ satisfies \eqref{dens-E} and that 
$\Gamma$-limit of $F^E_{\e_k}$ as $k\to+\infty$ exists for some $\e_k\to0$.
Then there exists the $\Gamma$-limit of $F^{E'}_{\e_k}$ as $k\to+\infty$, and the two $\Gamma$-limits are equal.
\end{theorem} 

\begin{remark}\rm An example in which \eqref{ci-ci} is satisfied is when $E$ is a $1$-periodic set and $E'=\Phi(E)$, where $\Phi:\mathbb R^d\to \mathbb R^d$ is a diffeomorphism with 
\begin{equation}
\lim_{|x|\to+\infty} |\Phi(x)-x|=0.
\end{equation}
This is a consequence of the fact that $|(E\triangle \Phi(E))\cap (k+[0,1]^d)|\to 0$ as $k\in\mathbb Z^d$ tends to $\infty$,
which is obtained by the periodicity of $E$ and  the Dominated Convergence Theorem.

Another example in which  \eqref{ci-ci} is satisfied trivially is by taking $E=\emptyset$ and $E'$ a set ``vanishing at infinity''; that is, such that
\begin{equation}
\lim_{R\to+\infty} {1\over R^d} |E'\cap B_R|=0.
\end{equation}

\end{remark}
 
 In order to refer to the setting of the previous sections, we introduce some coefficients as follows.
 We first set
\begin{equation}
a^E(x)= \begin{cases} 1 &\hbox{ if } x\in \mathbb R^d\setminus E\cr
0 &\hbox{ if } x\in E\,,
\end{cases}
\end{equation} 
so that we have
\begin{equation}\label{Esea}
F^E_\e(u)= \int_{\mathbb R^d} a^E\Big({x\over\e}\Big)|\nabla u|^2\dx
\end{equation} 
on $H^1(\mathbb R^d)$.

Moreover, for all $n\in\mathbb N$ we define
\begin{equation}
a^{E,n}(x)= \begin{cases} 1 &\hbox{ if } x\in \mathbb R^d\setminus E\cr
\frac1n &\hbox{ if } x\in E
\end{cases}\end{equation} 
and 
\begin{equation}\label{feta}
F^{E,n}_\e(u)= \int_{\mathbb R^d} a^{E,n}\Big({x\over\e}\Big)|\nabla u|^2\dx
\end{equation} 
on $H^1(\mathbb R^d)$.
We introduce analogous coefficients and functionals with $E'$ in the place of $E$.

Note that the integrands of \eqref{feta} belong to the class $\FF$ with $p=2$, $\alpha=\frac1n$, and $\beta=1$, while the integrands of \eqref{Esea} do not satisfy a growth condition from below on the perforation.

\begin{lemma}\label{lemmadomper} Let $E$ be a set satisfying extension properties 
\eqref{id-u} and \eqref{extDu} and density property \eqref{dens-E}. Then for every $n\in\mathbb N$ and for every sequence $\e_k\to 0$ we have 
\begin{eqnarray}\label{qd}
\Gamma\hbox{-}\liminf_{k\to +\infty} F^E_{\e_k}\le  \Gamma\hbox{-}\liminf_{k\to +\infty} F^{E,n}_{\e_k}\le \Big(1+\frac{C^2}n\Big)\Gamma\hbox{-}\liminf_{k\to +\infty} F^E_{\e_k}\\ \label{eqd}
\Gamma\hbox{-}\limsup_{k\to +\infty} F^E_{\e_k}\le  \Gamma\hbox{-}\limsup_{k\to +\infty} F^{E,\eta}_{\e_k}\le \Big(1+\frac{C^2}n\Big)\Gamma\hbox{-}\limsup_{k\to +\infty} F^E_{\e_k},
\end{eqnarray}
where $C$ is the constant in \eqref{extDu}.
\end{lemma}
\begin{proof}
We only prove \eqref{qd}, the proof of \eqref{eqd} being analogous. The first inequality is trivial; to prove the second inequality we fix $u$ and a sequence $u_k$ converging to $u$ in $L^2_{\rm loc}(\mathbb R^d)$ such that  
\[
\Gamma\hbox{-}\liminf_{k\to +\infty} F^E_{\e_k}(u)= \liminf_{k\to +\infty} F^E_{\e_k}(u_k).\]
We set $v_k= T_{\e_k} (u_k|_{\mathbb R^d\setminus \e_k E}) $  and note that, for each smooth connected $U\in\mathcal U$, up to subsequences, we can suppose that $v_k$ converges to some $v$ in $L^2(U)$. This follows from the bound on $\nabla v_k$, which implies pre-compactness up to a sequence of additive constants, and the fact that $v_k=u_k$ on $\mathbb  R^d\setminus \e_k E$, which implies that this sequence of constants is bounded, and hence pre-compactness without addition of a constant. Note that from $0=(u_k-v_k)1_{\mathbb  R^d\setminus \e_k E}$ and the convergence of $u_k$ to $u$ and $v_k$ to $v$ we deduce that $0=(u-v)\psi$ on each such $U$, where $\psi$ is a weak$^*$ limit in $L^\infty(\mathbb R^d)$ of $1_{\mathbb  R^d\setminus \e_k E}$.
Since $\psi>0$ almost everywhere by \eqref{dens-E} we deduce that $u=v$ almost everywhere in $\mathbb R^d$. Then 
\begin{eqnarray*} 
\Gamma\hbox{-}\liminf_{k\to +\infty} F^{E,n}_{\e_k}(u)&\le& \liminf_{k\to +\infty} F^{E,n}_{\e_k}(v_k)= \liminf_{k\to +\infty} \Big( F^{E}_{\e_k}(v_k) +{1\over n}\int_{\e_k E} |\nabla v_k|^2\dx\Big)\\
&\le& \liminf_{k\to +\infty}\Big(  F^{E}_{\e_k}(u_k) +{1\over n}\int_{\mathbb R^d} |\nabla v_k|^2\dx\Big)\\
&\le& \liminf_{k\to +\infty}\Big(  F^{E}_{\e_k}(u_k)  +{1\over n} C^2\int_{\mathbb R^d\setminus \e E} |\nabla u_k|^2\dx\Big)
\\ &=& \Big(1+\frac{C^2}n\Big)\,\Gamma\hbox{-}\liminf_{k\to +\infty} F^E_{\e_k}(u),
\end{eqnarray*}
which proves the claim.
\end{proof}

 Note  that the argument in the first part of the proof of the lemma shows that  if \eqref{dens-E} holds and  we have two such extension operators $T_{\e}$ and $\widetilde T_{\e}$,  if we have the strong convergence of $T_{\e} u_\e$ to $u$ and $\widetilde T_{\e} u_\e$ to $\widetilde u$ in $L^1(B_R)$ for all $R$, then we have
\begin{eqnarray*}
\psi u=\lim_{\e\to 0} 1_{B_R\setminus \e E} T_{\e} u_\e= \lim_{\e\to 0} 1_{B_R\setminus \e E} u_\e = \lim_{\e\to 0} 1_{B_R\setminus \e E} \widetilde T_{\e} u_\e= \psi \widetilde u\quad\hbox{ in $B_R$,}
\end{eqnarray*}
 which implies $\widetilde u= u$ almost everywhere since $\psi>0$ and $R$ is arbitrary.

\begin{proof}[Proof of Theorem {\rm\ref{Perfdo}}]
With fixed $n\in\mathbb N$, we define \[f^n_k(x,\xi)= a^{E,n}\big({x\over\e_k}\big)|\xi|^2\qquad \hbox{ and }\qquad g^n_k(y,\xi)= a^{E',n}\big({x\over\e_k}\big)|\xi|^2.\]
 We check that $f_k=f^n_k$ and $g_k=f^n_k$ satisfy \eqref{hypo} in Corollary \ref{Corgen}.
Since $B_{\rho/\e_k}(x/\e_k)\subset B_{(|x|+\rho)/\e_k}$, we deduce that for each fixed $\rho>0$ we have
\begin{eqnarray*}
&&\limsup_{k\to+\infty}{1\over\rho^d}\int_{B_\rho(x)}\sup_{|\xi|\le t} |f_k(y,\xi)-g_k(y,\xi)|\dy\\
&=& t^2\limsup_{k\to+\infty}{1\over\rho^d}\int_{B_\rho(x)} \Big|a^{E,n}({y\over\e_k})-a^{E',n}({y\over\e_k})\Big|\dy\\
&=& t^2\limsup_{k\to+\infty}{\e_k^d\over\rho^d}\int_{B_{\rho\over \e_k}({x\over\e_k})}  |a^{E,\eta}(y)-a^{E',\eta}(y)|dy\\
&\le&t^2{(|x|+\rho)^d\over \rho^d}\lim_{R\to+\infty} {1\over R^d}\int_{B_R} |a^{E,\eta}(y)-a^{E',\eta}(y)|dy\\
&=& t^2 {(|x|+\rho)^d\over \rho^d}\Big(1-{1\over n}\Big)\lim_{R\to+\infty}  {|(E\triangle E')\cap B_R|\over R^d}=0
\end{eqnarray*} by \eqref{ci-ci},
and \eqref{hypo} is satisfied.
Thanks to Theorem \ref{comp} the convergence hypothesis of Corollary \ref{Corgen} is satisfied up to passing to a subsequence. Hence, the integrands corresponding to the $\Gamma$-limits of the localized functionals $F^{E,n}_{\e_k}(\cdot, U)$ and $F^{E',n}_{\e_k}(\cdot, U)$ coincide.
Thanks to  Remark \ref{unb-comp}  this ensures the existence and equality of the corresponding $\Gamma$-limits on $\mathbb R^d$; that is, in the notation \eqref{feta}, that
\[
 \Gamma\hbox{-}\lim_{k\to+\infty} F^{E,n}_{\e_k}= \Gamma\hbox{-}\lim_{k\to+\infty} F^{E',n}_{\e_k}.\]
 By the existence of the limit for $E$ and Lemma \ref{lemmadomper}, we have
 \begin{eqnarray} \label{dis-0}\nonumber
\Gamma\hbox{-}\lim_{k\to+\infty} F^{E,n}_{\e_k}&\le& \Big(1+\frac{C^2}n\Big)\Gamma\hbox{-}\liminf_{k\to+\infty} F^E_{\e_k}\\
 &\le& \Big(1+\frac{C^2}n\Big)\Gamma\hbox{-}\limsup_{k\to+\infty} F^E_{\e_k}\le \Big(1+\frac{C^2}n\Big)\Gamma\hbox{-}\lim_{k\to+\infty}F^{E,n}_{\e_k}.
  \end{eqnarray} 
The same argument applied to $E'$ gives 
  \begin{eqnarray}  \label{dis-1}\nonumber
 \Gamma\hbox{-}\lim_{k\to+\infty} F^{E',n}_{\e_k}&\le& \Big(1+\frac{C^2}n\Big)\Gamma\hbox{-}\liminf_{k\to+\infty} F^{E'}_{\e_k}\\
 &\le& \Big(1+\frac{C^2}n\Big)\Gamma\hbox{-}\limsup_{k\to+\infty} F^{E'}_{\e_k}\le \Big(1+\frac{C^2}n\Big)\Gamma\hbox{-}\lim_{k\to+\infty}F^{E',n}_{\e_k} . \end{eqnarray}  
   We now set
\[
  F_0=\lim_{n\to+\infty}\big(\Gamma\hbox{-}\lim_{k\to+\infty} F^{E,n}_{\e_k}\big)= \lim_{n\to+\infty}\big(\Gamma\hbox{-}\lim_{k\to+\infty}  F^{E',n}_{\e_k}\big),\]
which exists since the functionals are decreasing in $n$. Letting $n\to +\infty$ in \eqref{dis-0} and \eqref{dis-1} we then have the existence of the limits, and equality
\[
\Gamma\hbox{-}\lim_{k\to+\infty} F^{E}_{\e_k}=\Gamma\hbox{-}\lim_{k\to+\infty} F^{E'}_{\e_k}= F_0.\]
Since the limit is independent of the subsequence used in the application of the compactness theorem, by the Urysohn property of $\Gamma$-convergence (see \cite{DM}) we conclude that the equality holds for the original sequence, and the claim. 
\end{proof}

\subsection{Convergence of solutions of minimum problems}
 Let $E\subset \mathbb R^d$ be a closed set satisfying \eqref{id-u}, \eqref{extDu} and such that there exists $c_0$ for which
 \begin{equation}\label{dens-E-s}
\liminf_{R\to+\infty}{|B_R\setminus E|\over R^d}\ge c_0>0,
\end{equation}
 which is a slightly stronger version of  
 \eqref{dens-E}. We suppose that the $\Gamma$-limit  $F_{\rm hom}^E$ of $ F^E_{\e}$ as $\e\to0$ exists.
 In particular this holds if $\mathbb R^d\setminus E$ is a connected Lipschitz domain.

 We fix $\lambda>0$, $f\in  L^2(\mathbb R^d)$,
and we study the convergence of solutions of minimum problems for energies of the form
\begin{equation}\label{Felambda}
F^\lambda_\e(u)=
\int_{\mathbb R^d\setminus \e E}(|\nabla u|^2+\lambda u^2- 2fu)\dx
\end{equation}
defined on $H^1(\mathbb R^d\setminus\e E)$. 

\begin{proposition}[Convergence of minimum problems]\label{pro-min}
 Let $E$ satisfy the conditions above, and let $T$ and $T_\e$ be the corresponding extension operators. Assume that there exists the weak limit $\psi$ of the characteristic functions $1_{\mathbb R^d\setminus \e E}$. Let $u_\e$ be the minimizer of $F^\lambda_\e$ in \eqref{Felambda}. Then $T_\e(u_\e)$ converge weakly in $H^1(\mathbb R^d)$ to the minimizer $u$ of 
\begin{equation}\label{Felambdahom}
F^\lambda_{\rm hom}(u)=F_{\rm hom}^E(u)+\int_{\mathbb R^d}(\lambda u^2-2 fu)\psi\dx.
\end{equation}
\end{proposition}

Theorem \ref{Perfdo} will guarantee the stability of this convergence with respect to perturbations satisfying \eqref{ci-ci}. Indeed, if we consider another set $E'$ satisfying the hypotheses of Theorem \ref{Perfdo}, \eqref{ci-ci} guarantees that $\psi$ is the same for $E$ and $E'$. 

We precede the proof of Proposition \ref{pro-min} by the representation of $F^\lambda_{\rm hom}$ as in the following lemma.

\begin{lemma}
There exists a measurable symmetric matrix $A_{\rm hom}:\mathbb R^d\to \mathbb M^{d\times d}$ satisfying  the boundedness and ellipticity conditions
\begin{equation}\label{ellip}
\frac1C |\xi|^2\le \langle A_{\rm hom}(x)\xi,\xi\rangle\le|\xi|^2\quad\hbox{for almost all $x\in\mathbb R^d$ and all $\xi\in\mathbb R^d$,}
\end{equation}
where $C$ is the constant in \eqref{extDu},
 such that the  $\Gamma$-limit of \eqref{quadFe} takes the form
\begin{equation}\label{limquad}
F_{\rm hom}^E(u)=\int_{\mathbb R^d}\langle A_{\rm hom}(x)\nabla u,\nabla u\rangle\dx
\end{equation}
for every $u\in H^1(\mathbb R^d)$.
\end{lemma}

\begin{proof}
The representation of \eqref{limquad} and the boundedness inequality in \eqref{ellip} can be proved as in 
\cite[Theorem 22.1]{DM}. In order to prove the ellipticity condition in \eqref{ellip}, we first show that 
\begin{equation}
\int_{\mathbb R^d}\langle A_{\rm hom}(x)\nabla u,\nabla u\rangle\dx\ge {1\over C} \int_{\mathbb R^d}|\nabla u|^2\dx
\end{equation}
for $u$ with compact support.
By the definition of $\Gamma$-convergence there exists a sequence $u_\e$ converging to $u$ in $L^2_{\rm loc}(\mathbb R^d)$ such that
$$
\lim_{\e\to 0}\int_{\mathbb R^d\setminus \e E}|\nabla u_\e|^2\dx= \int_{\mathbb R^d}
\langle A_{\rm hom}(x)\nabla u,\nabla u\rangle\dx.
$$
By a cut-off argument, multiplying by a function that is $1$ on the support of $u$, we can also assume that the functions $u_\e$ have support contained in a common compact set $K$, so that they strongly converge in $L^2(\mathbb R^d)$.
Let $R$ be large enough so that we have $|(B_R\setminus K)\setminus \e E|\ge c>0$ for all $\e$ small enough by \eqref{dens-E-s}. Since $T_\e(u_\e)=0$ on $(B_R\setminus K)\setminus \e E$ we can use Poincar\'e's inequality on $B_R\setminus K$ and, arguing as in the first part of the proof of Lemma \ref{lemmadomper}, deduce that  $T_\e(u_\e)$ converges to $0$ in $L^2(B_R\setminus K)$. 
Let now $\varphi\in C^\infty_c(\mathbb R^d)$ with $\varphi=1$ on $K$ and support in $\overline B_R$ and define $v_\e= \varphi T_\e(u_\e)$. Note that $v_\e\wto u$ in $H^1(\mathbb R^d)$, and by \eqref{extDu}
$$
\frac1C \int_{\mathbb R^d}|\nabla u|^2dx\le \liminf_{\e\to0}\frac1C \int_{\mathbb R^d}|\nabla v_\e|^2\dx\le \int_{\mathbb R^d}
\langle A_{\rm hom}(x)\nabla u,\nabla u\rangle\dx.
$$

Applying this last inequality to functions of the form $u(x)= w\big(\frac{x-x_0}\eta\big)$ and letting $\eta \to 0$ we deduce that
$$
\frac1C \int_{\mathbb R^d}|\nabla w|^2dx\le  \int_{\mathbb R^d}
\langle A_{\rm hom}(x_0)\nabla w,\nabla w\rangle\dx.
$$
for almost all $x_0$ and for all $w\in C^\infty_c(\mathbb R^d)$. A use of 
Parseval's identity then implies that 
$$
\frac1C \int_{\mathbb R^d}|\xi|^2|\widehat w|^2d\xi\le  \int_{\mathbb R^d}
\langle A_{\rm hom}(x_0)\xi,\xi\rangle|\widehat w|^2 d\xi.
$$
and  concludes the proof by the arbitrariness of $w$. 
\end{proof}

If $\mathbb R^d\setminus E$ is a connected Lipschitz $1$-periodic set, then  by  \cite{ACPDMP} $A_{\rm hom}$ is a constant matrix characterized by
$$
\langle A_{\rm hom}\xi,\xi\rangle=\min\Big\{ \int_{(0,1)^d\setminus E} |\xi+\nabla u|^2\dx: u\in H^1_{\rm loc}(\mathbb R^d) \hbox{ $1$-periodic}\Big\}.
$$

\begin{proof}[Proof of Proposition {\rm\ref{pro-min}}]
Let $u_\e$ be a sequence of minimizers for $F^\lambda_\e$, whose extensions $T_\e u_\e$ converge, up to subsequences, to some $\overline u$ in $L^2_{\rm loc} (\mathbb R^d)$.
Then we have 
\begin{eqnarray*}
&&\liminf_{\e\to0} F^\lambda_\e(u_\e)\\ &=&
\liminf_{\e\to0} \Bigl(\int_{\mathbb R^d\setminus \e E}
|\nabla u_\e|^2dx +\int_{\mathbb R^d} 1_{\mathbb R^d\setminus \e E}\Big(\sqrt\lambda u_\e- {1\over \sqrt\lambda} f\Big)^2dx-\int_{\mathbb R^d} 1_{\mathbb R^d\setminus \e E}{f^2\over \lambda}\dx\Bigr)\\
&\ge& \int_{\mathbb R^d}\langle A_{\rm hom}(x)\nabla\overline u,\nabla\overline  u\rangle\dx +\int_{\mathbb R^d} \psi \Big(\sqrt\lambda \overline u- {1\over \sqrt\lambda} f\Big)^2dx-\int_{\mathbb R^d}\psi {f^2\over \lambda}\dx
\\
&\ge& \int_{\mathbb R^d}\langle A_{\rm hom}(x)\nabla\overline u,\nabla\overline  u\rangle\dx+\int_{\mathbb R^d}\psi\big(\lambda \overline u^2-2 fu\big)\dx
\\
&\ge&\min\Big\{ \int_{\mathbb R^d}\langle A_{\rm hom}(x)\nabla u,\nabla u\rangle\dx +\int_{\mathbb R^d}\psi \big(\lambda u^2-2 fu\big)\dx: u\in H^1(\mathbb R^d)\Big\}
\end{eqnarray*}
by the $\Gamma$-convergence of $F^E_\e$ and the lower semicontinuity of the $L^2$ norm.
The converse inequality is obtained by noting that if we consider $u$ with compact support, up to a cut-off argument with functions that are $1$ on its support, we can find recovery sequences $u_\e$ for $F^E_{\rm hom}$ strongly converging in $L^2(\mathbb R^d)$, for which the additional terms pass to the limit. In conclusion, the minimizers for $F^\lambda_\e$ converge to the minimizer of
$F^\lambda_{\rm hom}$.\end{proof}

\noindent \textsc{Acknowledgements.}
 This article is based on work supported by the National Research Project PRIN 2022J4FYNJ  ``Variational methods for stationary and evolution problems with singularities and interfaces"
 funded by the Italian Ministry of University and Research. 
The first two authors are members of GNAMPA of INdAM.
The third author thanks SISSA for its hospitality. The authors thank the anonymous referees of {\em Annales de l'Institut Henri Poincar\'e C} for their many constructive comments that helped improving the text and also for pointing out some mathematical flaws in an earlier version of the manuscript.

 \bibliographystyle{abbrv}
\bibliography{Bra-DM-LB-2023}
 
\end{document}